\documentclass{article}
\usepackage{amsmath,amssymb,amsthm,amscd,dsfont}
\numberwithin{equation}{section} \allowdisplaybreaks
\begin{document}
\newtheorem{theorem}{Theorem}[section]
\newtheorem{defin}{Definition}[section]
\newtheorem{prop}{Proposition}[section]
\newtheorem{corol}{Corollary}[section]
\newtheorem{lemma}{Lemma}[section]
\newtheorem{rem}{Remark}[section]
\newtheorem{example}{Example}[section]
%\label{} %\ref{}
\title{Generalized para-K\"ahler manifolds}
\author{{\small by}\vspace{2mm}\\Izu Vaisman}
\date{}
\maketitle
{\def\thefootnote{*}\footnotetext[1]%
{{\it 2010 Mathematics Subject Classification: 53C15} .
}}
\begin{center} \begin{minipage}{12cm}
A{\footnotesize BSTRACT. We define a generalized almost para-Hermitian structure to be a commuting pair $(\mathcal{F},\mathcal{J})$ of a generalized almost para-complex structure and a generalized almost complex structure with an adequate non-degeneracy condition. If the two structures are integrable the pair is called a generalized para-K\"ahler structure. This class of structures contains both the classical para-K\"ahler structure and the classical K\"ahler structure. We show that a generalized almost para-Hermitian structure is equivalent to a triple $(\gamma,\psi,F)$, where $\gamma$ is a (pseudo) Riemannian metric, $\psi$ is a $2$-form and $F$ is a complex $(1,1)$-tensor field such that $F^2=Id,\gamma(FX,Y)+\gamma(X,FY)=0$. We deduce integrability conditions similar to those of the generalized K\"ahler structures and give several examples of generalized para-K\"ahler manifolds. We discuss submanifolds that bear induced para-K\"ahler structures and, on the other hand, we define a reduction process of para-K\"ahler structures.}
\end{minipage}
\end{center} \vspace{5mm}
\section{Introduction}
The framework of this paper is the $C^\infty$-category and the notation follows \cite{KN}, with the exception of the wedge product evaluation, where we follow Cartan (i.e., $\alpha\wedge\beta(X,Y)=\alpha(X)\beta(Y)-\alpha(Y)\beta(X)$ for $1$-forms, etc.).

Generalized geometry on a manifold $M$, as defined by Hitchin, is the geometry of structures of the big tangent bundle $\mathbf{T}M=TM\oplus T^*M$ endowed with the pairing metric
$$g(\mathcal{X},\mathcal{Y})=\frac{1}{2}(\alpha(Y)+\beta(X)),\; \mathcal{X}=(X,\alpha),\mathcal{Y}=(Y,\beta)\in TM\oplus T^*M$$
(e.g., see \cite{Galt}, which we also follow for terminology and notation).

The interpretation of the notion of a classical almost para-Hermitian structure in terms of generalized geometry leads us to define a generalized almost para-Hermitian structure to be a commuting pair $(\mathcal{F},\mathcal{J})$, where $\mathcal{J}$ is a generalized almost complex structure and $\mathcal{F}$ is a generalized almost para-complex structure \cite{{Galt},{IV}}, with an adequate non-degeneracy condition. Furthermore, if the two structures $(\mathcal{F},\mathcal{J})$ are integrable, the structure will be called a generalized para-K\"ahler structure. The integrability of $\mathcal{F}$ and $\mathcal{J}$ means that the eigenbundles of each of these endomorphisms are closed under the Courant bracket
$$
[(X,\alpha),(Y,\beta)]=([X,Y],L_X\beta-L_Y\alpha
+\frac{1}{2}d(\alpha(Y)-\beta(X)).$$
The Courant bracket \cite{C} replaces the Lie bracket in integrability conditions of generalized geometry.

In this paper, we apply the techniques of generalized K\"ahler geometry \cite{Galt} and obtain corresponding results for generalized para-K\"ahler structures.

We show that a generalized almost para-Hermitian structure is equivalent to a triple $(\gamma,\psi,F)$, where $\gamma$ is a (pseudo) Riemannian metric on $M$, $\psi$ is a $2$-form and $F\in End\,T^cM$ (the index $c$ denotes complexification) is a complex $(1,1)$-tensor field such that $F^2=Id,\gamma(FX,Y)+\gamma(X,FY)=0$. As a consequence, we show that the classical almost para-Hermitian and almost (pseudo) Hermitian manifolds have a generalized almost para-Hermitian structure. We also show that a generalized almost para-Hermitian structure is equivalent to a decomposition
$ \mathbf{T}^cM=\mathbf{H}_+\oplus\mathbf{H}_-\oplus \bar{\mathbf{H}}_+\oplus\bar{\mathbf{H}}_- $, where $\mathbf{H}_+\oplus\overline{\mathbf{H}}_+, \mathbf{H}_-\oplus\overline{\mathbf{H}}_-,\mathbf{H}_+\oplus\overline{\mathbf{H}}_-$ are maximal $g$-isotropic subbundles (the bar denotes complex conjugation).

We show that the integrability of the structure is equivalent to the closure of the subbundles $\mathbf{H}_\pm$ under the Courant bracket. Then, we prove that a generalized para-K\"ahler structure is characterized by the involutivity of the eigenbundles of the endomorphism $F$ together with either a certain expression of the covariant derivative $\nabla^\gamma F$ in terms of $d\psi$ or the equality $d\psi(X,Y,Z)={\rm i}d\omega(FX,FY,FZ)$ ($\nabla^\gamma$ is the Levi-Civita connection of $\gamma$ and ${\rm i}$ is the complex unit). In particular, if $d\psi=0$, we remain with the single condition $\nabla^\gamma F=0$. Among others, we construct a generalized para-K\"ahler structure on $\mathds{C}^2$ that projects to a structure of the complex $2$-torus and a generalized para-K\"ahler structure on a circle bundle over the classical paracomplex projective model \cite{GA}. Furthermore, we show that the complete lift of a generalized para-K\"ahler structure of a manifold $M$ to the tangent manifold $TM$ is a generalized para-K\"ahler structure of $TM$. Then, we briefly discuss a cohomology of the Dolbeault type.

In the last section, we discuss submanifolds that get an induced structure. These are the submanifolds $N$ of $M$ such that the tensor induced by $\gamma$ is non degenerate and $T^cN$ is invariant by $F$. Another characterization of these submanifolds is that the pullbacks of the subbundles $\mathbf{H}_\pm$ decompose $\mathbf{T}^cN$ in the way required by a generalized almost para-Hermitian structure. If the structure of $M$ is integrable, the induced structure of $N$ is integrable too.

Finally, we show that, if a structure $(\mathcal{F},\mathcal{J})$ is invariant by a Lie group $G$, a reduction process may exist and we define a momentum map for the complete lift of the $G$-invariant structure of $M$ to $TM$.
\section{Generalized almost para-Hermitian structures}
In classical differential geometry an almost paracomplex structure on a $2n$-dimensional manifold $M^{2n}$ is $F\in End\,TM$ with $F^2=Id$ and with $\pm1$-eigenbundles $S_\pm$ of dimension $n$. A pseudo-Riemannian metric $\gamma$ of $M$ is $F$-compatible if $\gamma(FX,Y)+\gamma(X,FY)=0$; then $S_\pm$ are isotropic and $\gamma$ has signature zero. The pair $(\gamma,F)$ is called an almost para-Hermitian structure. Furthermore, if the $2$-form $\omega(X,Y)=\gamma(X,FY)$ is closed, the structure is almost para-K\"ahler and if, in addition, $S_\pm$ are involutive, the structure is para-K\"ahler. With the musical isomorphisms $\flat_\gamma,\flat_\omega,\sharp_\gamma=\flat_\gamma^{-1},\sharp_\omega=-\flat_\omega^{-1}$, we have
\begin{equation}\label{musicprop} \begin{array}{l}\flat_\gamma\circ F=-F^*\circ\flat_\gamma,\, F\circ\sharp_\gamma=-\sharp_\gamma\circ F^*,\vspace*{2mm}\\ \flat_\omega=F^*\circ\flat_\gamma=-\flat_\gamma\circ F,\,\sharp_\omega=-\sharp_\gamma\circ F^*=F\circ\sharp_\gamma.\end{array}\end{equation}

These objects may be encoded in the endomorphisms
$\mathcal{F},\mathcal{J},\mathcal{H}\in End\,\mathbf{T}M$ given by the action of the matrices
\begin{equation}\label{matrixFJ}\mathcal{F}=\left(\begin{array}{lc}F&0\vspace*{2mm}\\ 0&-F^*\end{array} \right),\mathcal{J}=\left(\begin{array}{cc}0&\sharp_\omega\vspace*{2mm}\\ \flat_\omega&0\end{array} \right),
\mathcal{H}=\mathcal{F}\circ\mathcal{J}=\left(\begin{array}{cc}0&\sharp_\gamma\vspace*{2mm}\\ -\flat_\gamma&0\end{array}\right)\end{equation}
on columns $\left(\begin{array}{c}X\\ \alpha\end{array}\right)$, $X\in TM,\alpha\in T^*M$.
Then, $\mathcal{F},\mathcal{J}$ are skew-symmetric with respect to the pairing metric $g$
and $\mathcal{J}^2=-Id, \mathcal{F}^2=Id$, which means that $\mathcal{J}$ is a generalized almost complex structure \cite{Galt} and $\mathcal{F}$ is a generalized almost paracomplex structure \cite{IV}. Notice also the relations
\begin{equation}\label{propFGH} \mathcal{F}\circ\mathcal{J}=\mathcal{J}\circ\mathcal{F}=\mathcal{H},\,
\mathcal{F}\circ\mathcal{H}=\mathcal{H}\circ\mathcal{F}=\mathcal{J}, \,\mathcal{J}\circ\mathcal{H}=\mathcal{H}\circ\mathcal{J}=-\mathcal{F}.\end{equation}

The endomorphism $\mathcal{H}$ satisfies the properties
\begin{equation}\label{HGsym} \mathcal{H}^2=-Id,\; g(\mathcal{H}\mathcal{X},\mathcal{Y})=g(\mathcal{X},\mathcal{H}\mathcal{Y}),\end{equation} hence, it is not a generalized almost complex structure. We will say that $\mathcal{H}$ is a $g$-{\it symmetric big almost complex structure}. The $\pm{\rm i}$-eigenbundles $\mathbf{H},\bar{\mathbf{H}}$ of $\mathcal{H}$ are not $g$-isotropic, therefore, they cannot be closed under the Courant bracket and integrability of $\mathcal{H}$ is a non-sense condition.
Using $\mathbf{H}=im(Id-{\rm i}\mathcal{H})$, it follows that the non-degeneracy of $\gamma$ is equivalent to $\mathbf{H}\cap T^cM=0$.

This situation suggests the following definition.
\begin{defin}\label{genparaK} {\rm A {\it generalized almost para-Hermitian structure} is a commuting pair $(\mathcal{F},\mathcal{J})$, where $\mathcal{F}$ is a generalized almost paracomplex structure and $\mathcal{J}$ is a generalized almost complex structure, such that the symmetric bivector \begin{equation}\label{gammagen}
\gamma(\alpha,\beta)=-2g(\mathcal{F}(0,\alpha),\mathcal{J}(0,\beta))\end{equation} is non-degenerate.
If $\mathcal{J}$ is integrable, the structure is {\it generalized para-Hermitian}. If $\mathcal{F}$ is integrable, the structure is {\it generalized almost para-K\"ahler}. If both $\mathcal{F}$ and $\mathcal{J}$ are integrable, the structure is {\it generalized para-K\"ahler}.}\end{defin}

For any generalized almost para-Hermitian structure, the operator $\mathcal{H}=\mathcal{J}\circ\mathcal{F}$ is a $g$-symmetric big almost complex structure that satisfies the relations (\ref{propFGH}), (\ref{HGsym}) and we denote by $\mathbf{H}$ the ${\rm i}$-eigenbundle of $\mathcal{H}$.
\begin{prop}\label{gammanondeg} The tensor $\gamma$ defined by (\ref{gammagen}) is non-degenerate iff $\mathbf{H}\cap T^cM=0$.\end{prop}
\begin{proof}
We write down the following matrix representation of $\mathcal{F},\mathcal{J},\mathcal{H}$:
\begin{equation}\label{matrici} \mathcal{F}=\left(\begin{array}{lc}P&\sharp_\phi\vspace*{1mm}\\ \flat_\theta&-P^*\end{array} \right),\;\mathcal{J}=\left(\begin{array}{cc}A&\sharp_\pi\vspace*{2mm}\\ \flat_\sigma&-A^*\end{array} \right),\; \mathcal{H}=\left(\begin{array}{lc}Q&\sharp_\gamma\vspace*{1mm}\\ \flat_\nu&Q^*\end{array}\right),\end{equation} where
$\phi,\theta,\pi,\sigma$ are skew-symmetric, $\gamma$ is the tensor (\ref{gammagen}), $\nu\in\odot^2T^*M$ is a second symmetric tensor and
\begin{equation}\label{proptens} \begin{array}{l}
P^2=Id-\sharp_\phi\circ\flat_\theta,\,\sharp_\phi\circ P^*=P\circ\sharp_\phi,\,\flat_\theta\circ P=P^*\circ\flat_\theta,\vspace*{2mm}\\ A^2=-Id-\sharp_\pi\circ\flat_\sigma,\,\sharp_\pi\circ A^*=A\circ\sharp_\pi,\,\flat_\sigma\circ A=A^*\circ\flat_\sigma,\vspace*{2mm}\\ Q^2=-Id-\sharp_\gamma\circ\flat_\nu,\,\sharp_\gamma\circ Q^*=-Q\circ\sharp_\gamma,\,\flat_\nu\circ Q=-Q^*\circ\flat_\nu,\vspace*{2mm}\\ Q=P\circ A+\sharp_\phi\circ\flat_\sigma=A\circ P+\sharp_\pi\circ\flat_\theta,\vspace*{2mm}\\
\sharp_\gamma=A\circ\sharp_\pi-\sharp_\phi\circ A^*=A\circ\sharp_\phi+\sharp_\pi\circ P^*,\vspace*{2mm}\\  \flat_\nu=\flat_\theta\circ A-P^*\circ\flat_\sigma=\flat_\sigma\circ P-A^*\circ\flat_\theta;\end{array}\end{equation}
these relations ensure that $\mathcal{F},\mathcal{J},\mathcal{H}$ are structures of the required type and $\mathcal{H}=\mathcal{F}\circ\mathcal{J}$.

The condition $\mathbf{H}\cap T^cM=0$ is equivalent to $\mathbf{T}^cM=T^cM\oplus \mathbf{H}$, so that the corresponding projection onto $T^cM$ has an inverse $\tau:T^cM\rightarrow \mathbf{H}$. The latter must be of the form
$$X'+{\rm i}X''\,\mapsto\,(X'+{\rm i}X'',\alpha'+{\rm i}\alpha''),$$ where the right hand side is an ${\rm i}$-eigenvector of $\mathcal{H}$, equivalently:
$$\;(a)\hspace{5mm} QX'+\sharp_\gamma\alpha'=-X'',\;QX''+\sharp_\gamma\alpha''=X',$$
$$\;(b)\hspace{5mm} \flat_\nu X'+Q^*\alpha'=-\alpha'',\; \flat_\nu X''+Q^*\alpha''=\alpha'.$$
If the isomorphism $\tau$ exists, equations $(a)$ must define $\alpha',\alpha''$ uniquely, for any $X',X''$ and this happens iff $\gamma$ is non-degenerate.

If $\gamma$ is non-degenerate and if we solve $(a)$, we get
$$\alpha'=-\flat_\gamma(QX'+X''),\;\alpha''=\flat_\gamma(X'-QX'')$$
and we can check that the solutions also satisfy $(b)$ by using properties (\ref{proptens}). This shows that the obtained $\alpha',\alpha''$ yield an isomorphism $\tau$ as required.\end{proof}

Furthermore, put $\flat_\psi=-\flat_\gamma\circ Q$. Then, $\psi$ is a $2$-form and the values of $\alpha',\alpha''$ obtained above give the expression
\begin{equation}\label{tau} \tau(X)=(X,\flat_{\psi+{\rm i}\gamma}X),\; \forall X=X'+{\rm i}X''\in T^cM.
\end{equation}
The image of $\tau$ is $\mathbf{H}$, the image of the conjugate operator $\bar{\tau}$ is the $-{\rm i}$-eigenbundle $\bar{\mathbf{H}}$ of $\mathcal{H}$ and (\ref{tau}) is, in fact, one more way to express the operator $\mathcal{H}$. Another interesting consequence of formula (\ref{tau}) is
\begin{equation}\label{ggamma} \gamma(X,Y)=-\frac{{\rm i}}{2}g(\tau X,\tau Y),\end{equation}
where $X,Y\in T^cM$ and $\gamma$ is extended by complex linearity (we should have written $\gamma^{-1}$, but, we follow the custom of Riemannian geometry).
\begin{prop}\label{defstrcuF} The generalized almost para-Hermitian structures $(\mathcal{F},\mathcal{J})$ are in a one-to-one correspondence with triples $(\gamma,\psi,F)$, where $\gamma$ is a pseudo-Riemannian metric of $M$, $\psi$ is a $2$-form and $F\in End\,T^cM$ is a complex $(1,1)$-tensor field such that
\begin{equation}\label{Fgamma} F^2=Id,\,\gamma(FX,Y)=-\gamma(X,FY).\end{equation}
The same structures $(\mathcal{F},\mathcal{J})$ are in a one-to-one correspondence with triples $(\gamma,\psi,J\in End\,T^cM)$, where
\begin{equation}\label{Jgamma} J^2=-Id,
\gamma(JX,Y)=-\gamma(X,JY).\end{equation}
The tensors $F,J$ of a given structure $(\mathcal{F},\mathcal{J})$ are related by the equality $J={\rm i}F$.
\end{prop}
\begin{proof}
Since $\mathcal{F},\mathcal{H}$ commute, $\mathcal{F}$ preserves $\mathbf{H}$ and leads to a tensor $F\in End(T^cM)$ given by
\begin{equation}\label{FX} FX=pr_{T^cM}\mathcal{F}(\tau X).\end{equation}
Formulas (\ref{FX}), (\ref{ggamma}) imply (\ref{Fgamma}).
Similarly, $\mathcal{J}$ preserves $\mathbf{H}$ and produces the complex operator $JX=pr_{T^cM}\mathcal{J}(\tau X)$ with the properties (\ref{Jgamma}).

Generally, $F$ and $J$ are not real and, if we put
$$F=F_1+{\rm i}F_2,\,J=J_1+{\rm i}J_2$$ and use (\ref{matrici}), (\ref{tau}), we get
\begin{equation}\label{FXP} \begin{array}{l}
F_1=P-\sharp_\phi\circ\flat_\gamma\circ Q,\,
F_2=\sharp_\phi\circ\flat_\gamma,\vspace*{2mm}\\ J_1=A-\sharp_\pi\circ\flat_\gamma\circ Q,\,
J_2=\sharp_\pi\circ\flat_\gamma.\end{array}\end{equation}
Thus, $F$ is real iff $\phi=0$ and $J$ is real iff $\pi=0$.
The classical structure $(\gamma,F)$ is the case $\psi=0,\phi=0,\theta=0$ and, then, the operators on $T^cM$ are $F$ and $J={\rm i}\sharp_\omega\circ\flat_\gamma={\rm i}F$.

In real terms conditions (\ref{Fgamma}) become
\begin{equation}\label{F12gamma} \begin{array}{l}
F_1^2-F_2^2=Id, \,F_1\circ F_2+F_2\circ F_1=0,\vspace*{2mm}\\
\gamma(F_1X,Y)=-\gamma(X,F_1Y), \,\gamma(F_2X,Y)=-\gamma(X,F_2Y)\end{array}\end{equation}
and the same for the pair $(J_2,J_1)$ instead of $(F_1,F_2)$.

Conversely, the pair $(\gamma\in\odot^2T^*M,\psi\in\wedge^2T^*M)$, where $\gamma$ is non degenerate, allows us to reconstruct the big almost complex structure $\mathcal{H}$ by taking its eigenbundle $\mathbf{H}$ to be the image of $\tau$ given by (\ref{tau}). If we add the endomorphism $F$ of $T^cM$ that satisfies (\ref{Fgamma}), we are also able to reconstruct $\mathcal{F}$ on $\mathbf{H},\bar{\mathbf{H}}$ by lifting $F$ and its complex conjugation. The resulting $\mathcal{F}$ commutes with $\mathcal{H}$ (check on $\mathbf{H},\bar{\mathbf{H}}$). Finally, we will take $\mathcal{J}=\mathcal{H}\circ\mathcal{F}$ and $\mathcal{J}$ may be reconstructed from $J$ as $\mathcal{F}$ was from $F$. Therefore, we have
$$\left(\begin{array}{c}JX\vspace*{2mm}\\ \flat_{\psi+{\rm i}\gamma}JX\end{array}\right)=
\left(\begin{array}{cc}Q&\sharp_\gamma\vspace*{2mm}\\ \flat_\nu&Q^*\end{array}\right)
\left(\begin{array}{c}FX\vspace*{2mm}\\ \flat_{\psi+{\rm i}\gamma}FX\end{array}\right)
\;(X\in T^cM),$$ whence we deduce the required relation $J={\rm i}F$. This relation is equivalent to $J_1=-F_2,J_2=F_1$.

The entries of the respective matrices (\ref{matrici}) are determined by the real and imaginary part of the equalities
$$ \left(\begin{array}{lc}P&\sharp_\phi\vspace*{1mm}\\ \flat_\theta&-P^*\end{array} \right)
\left(\begin{array}{c}X\vspace*{1mm}\\ \flat_{\psi+{\rm i}\gamma}X\end{array} \right)
= \left(\begin{array}{c}F_1X+{\rm i}F_2X\vspace*{1mm}\\ \flat_{\psi+{\rm i}\gamma}(F_1X+{\rm i}F_2)X\end{array} \right),$$
$$\left(\begin{array}{lc}A&\sharp_\pi\vspace*{1mm}\\ \flat_\sigma&-A^*\end{array} \right)
\left(\begin{array}{c}X\vspace*{1mm}\\ \flat_{\psi+{\rm i}\gamma}X\end{array} \right)
= \left(\begin{array}{c}J_1X+{\rm i}J_2X\vspace*{1mm}\\ \flat_{\psi+{\rm i}\gamma}(J_1X+{\rm i}J_2)X\end{array}\right),$$
where we assume that $X\in TM$ is a real vector.
We shall also use $Q=-\sharp_\gamma\circ\flat_\psi,\flat_\nu=-\flat_\gamma\circ(Id+Q^2)$, which follows from the definition of $\psi$ and (\ref{proptens}). The results are
\begin{equation}\label{Qetc} \begin{array}{l}
P=F_1+F_2\circ Q,\, \sharp_\phi=F_2\circ\sharp_\gamma,\vspace*{2mm}\\ \flat_\theta=\flat_\psi\circ F_1-\flat_\gamma\circ F_2+(F_1^*+Q^*\circ F_2^*)\circ\flat_\psi,\vspace*{2mm}\\ A=J_1+J_2\circ Q,\, \sharp_\pi=J_2\circ\sharp_\gamma, \vspace*{2mm}\\ \flat_\sigma= \flat_\psi\circ J_1-\flat_\gamma\circ J_2+(J_1^*+Q^*\circ J_2^*)\circ\flat_\psi.\end{array}\end{equation}
\end{proof}

Manifolds $M$ endowed with a pair $(\gamma\in\odot^2T^*M,F\in End(T^cM))$, where $\gamma$ is non-degenerate and (\ref{Fgamma}) holds, belong to a class of manifolds discussed in \cite{Leg} under the name of almost-Hermitian manifolds in the enlarged sense. In \cite{Leg}, $\gamma$ too was allowed to be a complex tensor. In our case ($\gamma$ real) we call them {\it almost complexly-para-Hermitian} manifolds. They have two complementary fields of complex planes defined by the $\pm1$-eigenbundles of $F$, $S_\pm=im(Id\pm F)$, that are $\gamma$-isotropic, hence, both must have the same complex dimension $n=dim\,M/2$ and the non-degeneracy of $\gamma$ implies that $X\mapsto\flat_\gamma X$ defines isomorphisms $S_+\approx S_-^*,S_-\approx S_+^*$. It follows that, if we fix the metric $\gamma$, there exists a one-to-one correspondence between the complex tensor fields $F$ that satisfy (\ref{Fgamma}) and the decompositions $T^cM=S_+\oplus S_-$, where the terms
are maximally $\gamma$-isotropic subbundles.

For another remark, consider a $B$-{\it field transformation} ($B\in\wedge^2TM$) of the
generalized almost para-Hermitian structure $(\mathcal{F},\mathcal{J})$, i.e., the transformation $\mathcal{F}\mapsto\mathcal{F}^B=\mathcal{B}^{-1}\mathcal{F}\mathcal{B}, \mathcal{J}\mapsto\mathcal{J}^B=\mathcal{B}^{-1}\mathcal{J}\mathcal{B}$, where $\mathcal{B}=
\left(\begin{array}{cc}Id&0\vspace*{1mm}\\ \flat_B&Id\end{array}\right)$ \cite{{Galt},{IV}}. The result is again a generalized almost para-Hermitian structure. If the original structure is integrable and $B$ is closed, the transformed structure is integrable too. It follows easily that the $B$-field transformation preserves the tensors $\phi,\pi,\gamma$, while,
$$P\mapsto P'=P+\sharp_\phi\circ\flat_B,\, A\mapsto A'=A+\sharp_\pi\circ\flat_B,\, Q\mapsto Q'=Q+\sharp_\gamma\circ\flat_B.$$ Then, formula (\ref{FXP}) shows that the tensors $F,J$ are preserved, while, the definition of $\psi$ shows that $\psi\mapsto\psi'=\psi-B$.
\begin{example}\label{aprHerm} {\rm Let $(\gamma,J)$ be a classical almost (pseudo) Hermitian structure and take $\psi=0$. Since (\ref{Jgamma}) holds, we get a corresponding generalized almost para-Hermitian structure given by
\begin{equation}\label{exKahler}
\mathcal{F}=\left(\begin{array}{cc}0&-J\circ\sharp_\gamma\vspace*{2mm}\\ \flat_\gamma\circ J&0\end{array}\right),\; \mathcal{J}=\left(\begin{array}{cc}J&0\vspace*{2mm}\\ 0&-J^*\end{array}\right),\; \mathcal{H}=\left(\begin{array}{cc}0&\sharp_\gamma\vspace*{2mm}\\ -\flat_\gamma&0\end{array}\right).\end{equation}
Hence, we may see the generalized almost para-Hermitian structures as a bridge between the almost Hermitian and almost para-Hermitian structures.}\end{example}
\begin{example}\label{ex2Janti} {\rm Assume that $M$ has two (pseudo) Hermitian structures $(\gamma,J_\pm)$ with the same metric $\gamma$, such that the complex structures $J_+,J_-$ anti-commute and take any $\psi\in\wedge^2TM$. Then, the addition of a tensor field $F=\alpha J_++{\rm i}\beta J_-$, where $\alpha,\beta\in\mathds{R}, \beta^2-\alpha^2=1$, produces a generalized almost para-Hermitian structure on $M$. Similarly, if $(\gamma,K),(\gamma,L)$ are classical, almost pseudo-Hermitian and para-Hermitian, respectively, and if $K\circ L+L\circ K=0$, then, $\gamma$ together with a $2$-form $\psi$ and with $F=\alpha K+{\rm i}\beta L$ where $\alpha^2+\beta^2=1$, $\alpha,\beta\in\mathds{R}$, define a generalized, almost para-Hermitian structure with $F_1=\alpha K,F_2=\beta L$.}\end{example}
\begin{example}\label{proddir} {\rm The direct product $\tilde{M}=M\times M'$, where $M,M'$ are generalized almost para-Hermitian manifolds with the structures $(\mathcal{F},\mathcal{J}),(\mathcal{F}',\mathcal{J}')$, equivalently $(\gamma,\psi,F),(\gamma',\psi',F')$, has the generalized almost para-Hermitian structure $(\mathcal{F}\oplus\mathcal{F}',\mathcal{J}\oplus\mathcal{J}')$. This structure corresponds to the triple $(\gamma\oplus\gamma',\psi\oplus\psi',F\oplus F')$. The direct sum is just a notation to indicate that the terms act on corresponding components of the tangent bundle of $\tilde{M}$. For instance, if $\gamma$ is a pseudo-Euclidean metric of $\mathds{R}^{2n}$, of positive-negative inertia indices $(p,q)$, $p+q=2n$, $p\geq q$, we may write $\mathds{R}^{2n}=\mathds{R}^{2q}\times\mathds{R}^{2(n-q)}$, where the first factor has the standard para-Hermitian structure and the second factor has the standard Hermitian structure. The product structure is a generalized almost para-Hermitian structure on $\mathds{R}^{2n}$ ($2$-forms $\psi$ as needed may be added).}\end{example}

Consider a product $M\times M'$, where $(M,\gamma,F)$ is almost para-Hermitian and $(M',\gamma',{\rm i}F')$ is generalized almost para-Hermitian with a real tensor $F'$ (i.e., $(M',\gamma',F')$ is classical almost Hermitian). In this case the real and imaginary part of the product structure are $F_1=F\oplus\{0\},F_2=\{0\}\oplus F'$, which commute and we get $F_1\circ F_2=0$. Generally, a generalized almost para-Hermitian structure of a manifold $M$ such that $F_1\circ F_2=0$ will be called a {\it split structure}. The name is taken from \cite{AG} and it is motivated as follows. If $F_1\circ F_2=0$, then, (\ref{F12gamma}) implies $F_2\circ F_1=0$, $F_1^3=F_1,F_2^3=-F_2$. Furthermore, if $\Phi=F_1^2+F_2^2$, $\Phi^2=Id$ and we get a decomposition $TM=P\oplus Q$, where $P,Q$ are the $\pm1$-eigenbundles of $\Phi$ and $F_1^2|_P=Id,F_2^2|_Q=-Id$, $P=im\,F_1^2=im\,F_1=ker\,F_2$, $Q=im\,F_2^2=im\,F_2=ker\,F_1$.
\begin{example}\label{omogen} {\rm Let $M=G/H$, where $G$ is a connected Lie group that acts transitively and effectively on $M$ and $H$ is a closed subgroup, be a homogeneous, (pseudo) Riemannian space with the invariant metric $\gamma$. Let $x_0$ be the point defined by the unit $e$ of $G$. Using the construction indicated at the end of Example \ref{proddir}, we get a generalized almost para-Hermitian structure on $T_{x_0}M$ with a corresponding triple $(\gamma_{x_0},\psi_0,F_0)$.
Then, acting by $G$ on $F_0,\psi_0$ via the derived transformations we get invariant tensor fields $F,\psi$ of $M$ such that $F$ satisfies (\ref{Fgamma}) everywhere. The triple $(\gamma,\psi,F)$ produces an invariant, generalized almost para-Hermitian structure on $M$. It is known that the invariant tensors $\gamma,\psi,F$ may be identified with $ad\,\mathfrak{h}$-invariant tensors $\tilde{\gamma},\tilde{\psi},\tilde{F}$ ($\tilde{F}$ is complex) on the vector space $\mathfrak{g}/\mathfrak{h}$, where $\mathfrak{g},\mathfrak{h}$ are the Lie algebras of $G,H$, respectively. Accordingly, the invariant, generalized almost para-Hermitian structures of $M=G/H$ may be identified with generalized almost para-Hermitian structures of the vector space $\mathfrak{g}/\mathfrak{h}$ (the structures live on $(\mathfrak{g}/\mathfrak{h})\oplus(\mathfrak{g}/\mathfrak{h})^*$).}\end{example}

We end this section by proving the following result.
\begin{prop}\label{propdedesc} A generalized almost para-Hermitian structure on $M$ is equivalent with a decomposition
\begin{equation}\label{descptparaK} \mathbf{T}^cM=\mathbf{H}_+\oplus\mathbf{H}_-\oplus \bar{\mathbf{H}}_+\oplus\bar{\mathbf{H}}_-, \end{equation}
where $\mathbf{H}_+\oplus\overline{\mathbf{H}}_+, \mathbf{H}_-\oplus\overline{\mathbf{H}}_-,\mathbf{H}_+\oplus\overline{\mathbf{H}}_-$ are maximal $g$-isotropic subbundles.\end{prop}
\begin{proof} Start with a structure $(\mathcal{F},\mathcal{J})$ and denote by $\mathbf{F}_\pm$ the $\pm1$-eigenbundles of $\mathcal{F}$ and by $\mathbf{J},\bar{\mathbf{J}}$ the $\pm {\rm i}$-eigenbundles of $\mathcal{J}$. The commutation properties (\ref{propFGH}) ensure that the projections of a vector of $\mathbf{F}_\pm,\mathbf{J}$ on $\mathbf{H},\bar{\mathbf{H}}$ belong to $\mathbf{F}_\pm,\mathbf{J}$, respectively, and the projections of a vector of $\mathbf{H}$ on $\mathbf{F}_\pm,\mathbf{J}$ belong to $\mathbf{H}$.
This leads to the existence of the following decompositions
\begin{equation}\label{descH} \begin{array}{l}
\mathbf{F}_\pm=(\mathbf{F}_\pm\cap \mathbf{H})\oplus(\mathbf{F}_\pm\cap\bar{\mathbf{H}}),\,\mathbf{J}=(\mathbf{J}\cap \mathbf{H})\oplus(\mathbf{J}\cap\bar{\mathbf{H}}),\vspace*{2mm}\\ \mathbf{H}=(\mathbf{F}_+\cap \mathbf{H})\oplus(\mathbf{F}_-\cap \mathbf{H}),\, \mathbf{H}=(\mathbf{H}\cap\mathbf{J})\oplus (\mathbf{H}\cap\bar{\mathbf{J}}).\end{array}\end{equation}
Moreover, by looking at the properties of a vector in the corresponding intersection, we get
\begin{equation}\label{=cap}
\mathbf{F}_+\cap\mathbf{H}=\mathbf{F}_+\cap\mathbf{J}= \mathbf{H}\cap\mathbf{J},\,
\mathbf{F}_-\cap\mathbf{H}=\mathbf{F}_-\cap\mathbf{\bar{J}} =\mathbf{H}\cap\mathbf{\bar{J}}.\end{equation}
Then, if we define $\mathbf{H}_\pm=\mathbf{F}_\pm\cap\mathbf{H}$, equalities (\ref{descH}) give us (\ref{descptparaK}) and the subbundles of the conclusion are exactly $\mathbf{F}_\pm,\mathbf{J}$, hence, the conclusion holds.

Conversely, a decomposition (\ref{descptparaK}) produces endomorphisms $\mathcal{F},\mathcal{J}$ such that
$$\begin{array}{l} \mathcal{F}_{\mathbf{H}_+}=\mathcal{F}_{\overline{\mathbf{H}}_+}=Id,
\mathcal{F}_{\mathbf{H}_-}=\mathcal{F}_{\overline{\mathbf{H}}_-}=-Id,\vspace*{2mm}\\ \mathcal{J}_{\mathbf{H}_+}=\mathcal{J}_{\overline{\mathbf{H}}_-}={\rm i}Id,
\mathcal{J}_{\mathbf{H}_-}=\mathcal{J}_{\overline{\mathbf{H}}_+}=-{\rm i}Id.
\end{array}$$
These endomorphisms $\mathcal{F},\mathcal{J}$ are real, commute and satisfy the conditions $\mathcal{F}^2=Id,\mathcal{J}^2=-Id$. Since the subbundles $\mathbf{F}_+
=\mathbf{H}_+\oplus\overline{\mathbf{H}}_+, \mathbf{F}_-
=\mathbf{H}_-\oplus\overline{\mathbf{H}}_-, \mathbf{J}
=\mathbf{H}_+\oplus\overline{\mathbf{H}}_-$ were required to be maximal $g$-isotropic (which also implies $dim\,\mathbf{H}_+=dim\,\mathbf{H}_-$), the pair $(\mathcal{F},\mathcal{J})$ is a generalized almost para-Hermitian structure of $M$.
\end{proof}
\section{Generalized para-K\"ahler manifolds}
In this section we investigate the integrability conditions of a generalized almost para-Hermitian structure $(\mathcal{F},\mathcal{J})$. These conditions can be obtained in the same way as for generalized K\"ahler manifolds \cite{{Galt},{VgenS}}.
\begin{theorem}\label{thintegr} A generalized almost para-Hermitian structure $(\mathcal{F},\mathcal{J})$ is integrable iff the corresponding subbundles $\mathbf{H}_\pm$ are closed under Courant brackets. Furthermore, if $(\gamma,\psi,F)$ is the equivalent triple of tensors of the structure, $(\mathcal{F},\mathcal{J})$ is a generalized para-K\"ahler structure iff the (complex) eigenbundles of $F$ are involutive and
\begin{equation}\label{paraKprinF} (\nabla^\gamma_XF)(Y)=\frac{{\rm i}}{2}\sharp_\gamma[i(X)i(FY)d\psi+(i(X)i(Y)d\psi)\circ F],
\end{equation} where $\nabla^\gamma$ is the Levi-Civita connection of the metric $\gamma$.\end{theorem}
\begin{proof}
We will use decomposition (\ref{descptparaK}). Equalities (\ref{=cap}) show that, if $\mathcal{F},\mathcal{J}$ are both integrable, the subbundles $\mathbf{H}_\pm$ are closed under the Courant bracket. Conversely, assume that $\mathbf{H}_\pm$ are closed under the Courant bracket. Consider the following general property \cite{C}
$$ \begin{array}{c}(pr_{TM}\mathcal{Z})(g(
\mathcal{X},\mathcal{Y}))=g([ \mathcal{Z},\mathcal{X}],\mathcal{Y})
+g(\mathcal{X},[ \mathcal{Z},\mathcal{Y}])\vspace{2mm}\\
+\frac{1}{2}(pr_{TM}
\mathcal{X})(g( \mathcal{Z},\mathcal{Y})) +\frac{1}{2}(pr_{TM}
\mathcal{Y})(g( \mathcal{Z},\mathcal{X})).
\end{array}$$ For
$(\mathcal{X},\mathcal{Y},\mathcal{Z})\mapsto
(\mathcal{Z},\bar{\mathcal{Y}},\mathcal{X})$, where $\mathcal{X},\mathcal{Y},\mathcal{Z}\in\mathbf{H}_\pm$,
this property and the $g$-isotropy of $\mathbf{F}_\pm=\mathbf{H}_\pm\oplus\bar{\mathbf{H}}_\pm$ yield
$[\mathcal{X},\bar{\mathcal{Y}}]\perp_g\mathbf{H}_\pm$. By conjugation, and changing the role of $\mathcal{X},\mathcal{Y}$, we also get $[\mathcal{X},\bar{\mathcal{Y}}]\perp_g\bar{\mathbf{H}}_\pm$, therefore $[\mathcal{X},\bar{\mathcal{Y}}]\perp_g\mathbf{F}_\pm$ and, since $\mathbf{F}_\pm$ is maximally isotropic, $[\mathcal{X},\bar{\mathcal{Y}}]\in\mathbf{F}_\pm$. This proves the closure of $\mathbf{F}_\pm$ under Courant brackets, i.e., the integrability of $\mathcal{F}$. The closure of $\mathbf{J}$ under Courant brackets follows by the same computations for arguments $\mathcal{X},\mathcal{Z}\in\mathbf{H}_+,
\mathcal{Y}\in\mathbf{H}_-$ and $\mathcal{X}\in\mathbf{H}_+,
\mathcal{Z},\mathcal{Y}\in\mathbf{H}_-$ using the maximal isotropy of $\mathbf{J}$ and the already proven result for $\mathbf{F}_\pm$. Therefore $\mathcal{J}$ is integrable too.

For the second part of the theorem, straightforward computations (as in\cite{{Galt},{VgenS}}) yield the Courant bracket
\begin{equation}\label{auxintegr} [\tau X,\tau Y]=\tau[X,Y] +(0,i(Y)i(X)d\psi+{\rm i}(L_Xi(Y)-i(X)L_Y)(\gamma)),\end{equation}
where $X,Y\in T^cM$.

Therefore, $pr_{T^cM}[\tau X,\tau Y]=[X,Y]$ and, since $\mathbf{H}_\pm=\tau(S_\pm)$, where $S_\pm$ are the (complex) $\pm1$-eigenbundles of $F$, involutivity of $S_\pm\subseteq T^cM$ is a necessary condition for the integrability of the structure $(\mathcal{F},\mathcal{J})$. Moreover, (\ref{auxintegr}) shows that the necessary and sufficient conditions are involutivity of $S_\pm$ plus the two equalities
\begin{equation}\label{condsuf2}(L_Xi(Y)-i(X)L_Y)(\gamma))={\rm i}i(X)i( Y)d\psi, \;X,Y\in S_\pm.\end{equation}

We shall express these conditions in terms of the Levi-Civita connection $\nabla^\gamma$ of $\gamma$. Involutivity of $S_\pm$ is equivalent to the vanishing of the (complex) Nijenhuis tensor
$$N_F(X,Y)=[FX,FY]-F[FX,Y]-F[X,FY]+F^2[X,Y]=0,\;\;X,Y\in T^cM,$$
where the brackets are Lie brackets. Thus, we can express them by the (torsionless) Levi-Civita connection $\nabla^\gamma$ of the metric $\gamma$. The result is
\begin{equation}\label{NijFgamma} (\nabla^\gamma_{FX}F)(Y)-(\nabla^\gamma_{FY}F)(X) =F[(\nabla^\gamma_{X}F)(Y)-(\nabla^\gamma_{Y}F)(X)].
\end{equation}

Then, if we evaluate (\ref{condsuf2}) on $Z\in T^cM$ and express the Lie derivatives using $\nabla^\gamma$, (\ref{condsuf2}) becomes
\begin{equation}\label{condsuf21} \gamma(X,\nabla^\gamma_ZY)-\gamma(Y,\nabla^\gamma_ZX)
={\rm i}d\psi(X,Y,Z), \;X,Y\in S_\pm.\end{equation}

Furthermore, let $\nabla$ be a connection on $T^cM$ such that $\nabla\gamma=0,\nabla F=0$. Such connections exist. For instance, we may take $$\nabla_XY=pr_{S_+}\nabla^\gamma_X(pr_{S_+}Y)+pr_{S_-}\nabla^\gamma_X(pr_{S_-}Y).$$ $\nabla$ preserves $S_\pm$, which is equivalent to $\nabla F=0$, and $\nabla\gamma=0$ follows by taking into account $\gamma|_{S_\pm}=0,\nabla^\gamma\gamma=0$. If $\Theta=\nabla-\nabla^\gamma$, $\Theta$ satisfies the conditions
\begin{equation}\label{auxTheta}\begin{array}{l}\gamma(\Theta(X,Y),Z) +\gamma(Y,\Theta(X,Z))=0,\vspace*{2mm}\\ (\nabla^\gamma_XF)(Y)=F\Theta(X,Y)-\Theta(X,FY).\end{array}\end{equation}
Using (\ref{auxTheta}) and the $\gamma$-isotropy of $S_\pm$,  condition (\ref{condsuf21}) becomes
\begin{equation}\label{condsuf22} \gamma(Y,\Theta(Z,X))
=\frac{{\rm i}}{2}d\psi(X,Y,Z), \;X,Y\in S_\pm.\end{equation}

In (\ref{condsuf22}) we may replace $X,Y$ by $X\pm FX,Y\pm FY$, where the new arguments $X,Y\in T^cM$ are arbitrary vector fields. Then, using (\ref{auxTheta}) again, we get
$$\begin{array}{l} \gamma(X,(\nabla^\gamma_ZF)(FY))+ \gamma(X,(\nabla^\gamma_ZF)(Y))\vspace*{2mm}\\ = \frac{{\rm i}}{2}[d\psi(X,Y,Z)+d\psi(FX,FY,Z)+d\psi(X,FY,Z)+d\psi(FX,Y,Z)],\vspace*{2mm}\\	 \gamma(X,(\nabla^\gamma_ZF)(FY))- \gamma(X,(\nabla^\gamma_ZF)(Y))\vspace*{2mm}\\ = \frac{{\rm i}}{2}[d\psi(X,Y,Z)+d\psi(FX,FY,Z)-d\psi(X,FY,Z)-d\psi(FX,Y,Z)].\end{array}$$
These conditions may be replaced by their sum and difference, which turn out to be the same condition applied to $Y$ and $FY$, respectively. The single remaining condition is
\begin{equation}\label{paraKF}
\gamma(X,(\nabla^\gamma_ZF)(Y))= \frac{{\rm i}}{2}[d\psi(X,FY,Z)+d\psi(FX,Y,Z)],
\end{equation}
which is equivalent to (\ref{paraKprinF}).
\end{proof}

Notice that (\ref{Fgamma}) implies
\begin{equation}\label{nablaomegaF}
\gamma(X,(\nabla^\gamma_ZF)(Y))=(\nabla^\gamma_Z\omega)(X,Y),\end{equation}
where $\omega(X,Y)=\gamma(X,FY)$. This yields one more expression of the integrability condition (\ref{paraKF}).
\begin{corol}\label{integrcupsizero} A generalized almost para-Hermitian structure associated to a triple $(\gamma,\psi,F)$ where $d\psi=0$ is generalized para-K\"ahler iff $\nabla^\gamma F=0$.\end{corol}
\begin{proof} For $d\psi=0$, (\ref{paraKprinF}) becomes $\nabla^\gamma F=0$. Furthermore, $\nabla^\gamma F=0$ and (\ref{NijFgamma}) imply the involutivity of $S_\pm$.
\end{proof}
\begin{rem}\label{obsconexpsi} {\rm In the general case, formula $\nabla_XY=\nabla^\gamma_XY-\frac{{\rm i}}{2}\sharp_\gamma[i(X)i(Y)d\psi]$ defines a new metric connection and (\ref{paraKprinF}) is equivalent to $\nabla F=0$.}\end{rem}

Another form of the integrability conditions is given by the following theorem.
\begin{theorem}\label{th2integr} The generalized almost para-Hermitian structure associated to a triple $(\gamma,\psi,F)$ is integrable iff the eigenspaces $S_\pm$ of $F$ are involutive and $\psi$ satisfies the condition
\begin{equation}\label{necint} d\psi(X,Y,Z)={\rm i}d\omega(FX,FY,FZ).\end{equation}
\end{theorem}
\begin{proof}
If we evaluate the $1$-form $\flat_\gamma Z$ on the left hand side of (\ref{NijFgamma}) and take into account (\ref{nablaomegaF}), the involutivity condition of the subbundles $S_\pm$ becomes
\begin{equation}\label{involomega}
(\nabla^\gamma_{FX}\omega)(Z,Y)-(\nabla^\gamma_{FY}\omega)(Z,X) =(\nabla^\gamma_{Y}\omega)(FZ,X)-(\nabla^\gamma_{X}\omega)(FZ,Y).\end{equation}
Changing $Z$ to $FZ$ in (\ref{involomega}) and using the formula
\begin{equation}\label{difcunabla}
d\omega(X,Y,Z)=\sum_{Cycl(X,Y,Z)}(\nabla^\gamma_X\omega)(Y,Z),\end{equation}
we see that the involutivity of $S_\pm$ is equivalent to
\begin{equation}\label{involomega2}
d\omega(X,Y,Z)=(\nabla^\gamma_{Z}\omega)(X,Y)-(\nabla^\gamma_{FX}\omega)(Z,FY)
+(\nabla^\gamma_{FY}\omega)(Z,FX).\end{equation}

Notice that (\ref{Fgamma}) and $\nabla^\gamma(F^2)=0$ imply
\begin{equation}\label{propnablaomega}\begin{array}{c}
(\nabla^\gamma_{Z}\omega)(FX,Y)=\gamma(FX,(\nabla^\gamma_ZF)Y)= -\gamma(X,F(\nabla^\gamma_ZF)Y) \vspace*{2mm}\\ =\gamma(X,(\nabla^\gamma_ZF)FY)=(\nabla^\gamma_{Z}\omega)(X,FY),
\end{array}\end{equation}
whence, also, $(\nabla^\gamma_{Z}\omega)(FX,FY)=(\nabla^\gamma_{Z}\omega)(X,Y)$.

Furthermore, changing the arguments to $FX,FY,Z$ in (\ref{involomega2}), the result becomes
\begin{equation}\label{auxomega}2(\nabla^\gamma_{Z}\omega)(X,Y) =d\omega(X,Y,Z)+d\omega(FX,FY,Z), \end{equation}
which, therefore, is another expression of the involutivity of $S_\pm$. This condition has the following consequences. Take $X,Y,Z$ either all in $S_+$ or all in $S_-$. Then (\ref{auxomega}) yields $d\omega(X,Y,Z)=(\nabla^\gamma_{Z}\omega)(X,Y)$ and, if we add the two cyclic permutation, the result is $d\omega(X,Y,Z)=0$. This allows us to check the following equality for general arguments $X,Y,Z\in T^cM$:
\begin{equation}\label{omega3F} d\omega(FX,Y,Z)+d\omega(X,FY,Z)+d\omega(X,Y,FZ) =-d\omega(FX,FY,FZ)\end{equation} (check for each possible combination of arguments that are eigenvectors of $F$).

Now, modulo (\ref{nablaomegaF}) and (\ref{auxomega}), the integrability condition (\ref{paraKF}) becomes
\begin{equation}\label{intaux} d\omega(X,Y,Z)+d\omega(FX,FY,Z)=
{\rm i}[d\psi(X,FY,Z)+d\psi(FX,Y,Z)].\end{equation}
Furthermore, notice the following direct consequence of (\ref{paraKF}), (\ref{difcunabla}):
\begin{equation}\label{ciclic2} d\omega(X,Y,Z)={\rm i}[d\psi(FX,Y,Z)+d\psi(X,FY,Z)+d\psi(X,Y,FZ)].\end{equation}
If (\ref{ciclic2}) is inserted into (\ref{intaux}) and $Z$ is replaced by $FZ$, we get the equality (\ref{necint}), which, therefore, is a necessary condition of integrability.
We will show that condition (\ref{necint}) is also sufficient for integrability. Firstly, we notice that (\ref{necint}) implies $d\psi(X,Y,Z)=0,\,\forall X,Y,Z\in S_\pm$ and, checking on arguments that are eigenvectors of $F$, we get
\begin{equation}\label{3dpsi} d\psi(FX,Y,Z)+d\psi(X,FY,Z)+d\psi(X,Y,FZ)=-d\psi(FX,FY,FZ).
\end{equation}
If we calculate the left hand side of (\ref{intaux}) modulo (\ref{necint}) and take into account (\ref{3dpsi}), the result is exactly the right hand side of (\ref{intaux}) and we are done. \end{proof}
\begin{corol}\label{integrcupsizero2} A generalized almost para-Hermitian structure associated to a triple $(\gamma,\psi,F)$ where $d\psi=0$ is generalized para-K\"ahler iff the eigenbundles $S_\pm$ of $F$ are involutive and the $2$-form $\omega$ is closed.\end{corol}
\begin{proof} This is a straightforward consequence of Theorem \ref{th2integr}. \end{proof}
\begin{example}\label{exclas} {\rm If $(\gamma,F),(\gamma,J)$ are a classical para-K\"ahler and classical K\"ahler structure of a manifold $M$, the corresponding generalized structures (\ref{matrixFJ}), (\ref{exKahler}) are integrable. Similarly, if the quadruple $(\gamma,\psi,J_\pm)$, with an anti-commuting pair $J_\pm$ and a closed form $\psi$, defines a generalized K\"ahler structure, the triple $(\gamma,\psi,F)$, where $F=\alpha J_++{\rm i}\beta J_-$ ($\alpha,\beta\in\mathds{R}, \beta^2-\alpha^2=1$) (see Example \ref{ex2Janti}) defines a generalized para-K\"ahler structure. Indeed, in the considered case, the generalized K\"ahler condition is $\nabla^\gamma J_\pm=0$ \cite{{Galt},{IV}}, which implies $\nabla^\gamma F=0$.
The previous conclusion also holds for the second case of Example \ref{ex2Janti}, i.e., $F=\alpha K+{\rm i}\beta L$.

We shall give the following concrete example. Take $M=\mathds{C}^2$ with the coordinates
$$z^1=x^1+{\rm i}y^1,\,z^2=x^2+{\rm i}y^2,\;x^1,x^2,y^1,y^2\in\mathds{R}.$$
Put $$\gamma={\rm i}(dz^1\odot d\bar{z}^2-dz^2\odot d\bar{z}^1)=2(dx^1\odot dy^2-dx^2\odot dy^1),$$
where $\odot$ is the symmetrized tensor product, and consider the real $(1,1)$-tensor fields
$$K\frac{\partial}{\partial z^a}={\rm i}\frac{\partial}{\partial z^a},\;
L\frac{\partial}{\partial x^a}=\frac{\partial}{\partial x^a}, L\frac{\partial}{\partial y^a}=-\frac{\partial}{\partial y^a},\;a=1,2.$$ It is easy to check that $K,L$ anti-commute, $(\gamma,K)$ is a K\"ahler structure and $(\gamma,L)$ is a para-K\"ahler structure. The K\"ahler form of $(\gamma,K)$ and the para-K\"ahler (fundamental) form of $(\gamma,K)$ are
$$\omega_K=dx^1\wedge dx^2+dy^1\wedge dy^2,\;\omega_L=dx^2\wedge dy^1+dy^2\wedge dx^1.$$
Hence, $(\gamma,\psi,F=\alpha K+{\rm i}\beta L)$, with $d\psi=0$ and real coefficients $\alpha,\beta$ such that $\alpha^2+\beta^2=1$, defines a generalized para-K\"ahler structure.
Furthermore, if we quotientize by a lattice, the structure descends to the $4$-dimensional torus.
On the other hand, we can obtain a generalized almost para-Hermitian structure on the Hopf manifold $(\mathds{C}^2\setminus\{0\})/\Delta_\lambda$, where $\Delta_\lambda$ is the group generated by the transformation $(z^1,z^2)\mapsto(\lambda z^1,\lambda z^2)$ with a real constant $\lambda\neq0,1$. Namely, consider the same $F$ as above and the metric $\gamma/(z^1\bar{z}^1
+z^2\bar{z}^2)$ on $(\mathds{C}^2\setminus\{0\})$ and go down to the quotient manifold. This structure is not integrable even though the $\pm1$-eigenbundles of $F$ are involutive.} \end{example}
\begin{example}\label{prodK} {\rm The product $M\times M'$ of a classical para-K\"ahler manifold $M$ and a classical K\"ahler manifold $M'$ with the structure defined in Example \ref{proddir}, where we assume that $\psi,\psi'$ are closed, is a generalized para-K\"ahler manifold. Indeed, it satisfies the condition $\nabla^\gamma F=0$.

More generally, assume that $p:M\rightarrow N$ is a locally trivial fiber bundle with fiber $S$ and structural group $G$. Assume that $N$ has a classical para-K\"ahler structure $(\gamma_N,F_N)$ and $S$ has a $G$-invariant K\"ahler structure $(\gamma_S,F_S)$. Then, $M$ has a generalized para-K\"ahler structure obtained by gluing-up the product structures of the domains $U_\alpha\times S$, where $\{U_\alpha\}$ is a covering of $N$ by local trivializing neighborhoods. The same holds if $N$ is K\"ahler and $S$ has a $G$-invariant para-K\"ahler structure.

A concrete example is that of the manifold $M=(E\times(\mathds{C}^m\setminus\{0\}))/\mathds{R}$, where $E=\{(x^k,y^k)\,/\,\sum_{k=1}^nx^ky^k=1\}\subset\mathds{R}^{2n}$ and the additive group $\mathds{R}$ acts by
\begin{equation}\label{action}
(x^k,y^k,z^u)\mapsto(e^tx^k,e^{-t}y^k,e^{it}z^u),\;\;t\in\mathds{R},
\end{equation}
where $(z^u)$ are the natural coordinates of $\mathds{C}^m$.
Then,  $(x^k,y^k,z^u)\mapsto(x^k,y^k)$ defines a submersion $p:M\rightarrow\mathds{B}\mathbf{P}^{n-1}$, where $\mathds{B}\mathbf{P}^{n-1}$ is the paracomplex projective model, a known example of a para-K\"ahler manifold \cite{{GA},{VJMP}}.
Hence, $M$ is a locally trivial bundle over $\mathds{B}\mathbf{P}^{n-1}$ with fiber $\mathds{C}^m\setminus\{0\}$ and structural group $S^1$. Since the latter preserves the classical K\"ahler structure of $\mathds{C}^m\setminus\{0\}$, $M$ has a generalized para-K\"ahler structure. Formula (\ref{action}) shows that the pair $(\gamma,F)$ of this structure is induced by the tensor fields
\begin{equation}\label{modelF} \begin{array}{l}
\gamma=\sum_{k=1}^ndx^k\otimes dy^k +\sum_{u=1}^mdz^u\otimes d\bar{z}^u,\vspace*{2mm}\\ dx^k\circ F=dx^k,\,dy^k\circ F=-dy^k,\,dz^u\circ F=dz^u,\,d\bar{z}^u\circ F=-d\bar{z}^u\end{array}
\end{equation}
and any closed $2$-form $\psi$ may be added.

Notice that $dim\,M=2n+2m-2$ and the differentiable structure of $M$ is given by the following, local, non-homogeneous coordinates. Cover $M$ by the domains $U_{(k)}|_{k=1}^n=pr_M\{(x^k,y^k,z^u)\in E\times(\mathds{C}^{2m}\setminus\{0\})\,/\,x^k\neq0\}$ (obviously, $x^k$ cannot all vanish at the same point) and define the coordinates on $U_{(k)}$ by
$$(X^h_{(k)}=x^h/|x^k|,Y^h_{(k)}=y^h|x^k|,Z^u_{(k)}=e^{-iln|x^k|}z^u)\,(h\neq k)$$
($y^k$ is determined by the equation $\sum_{h=1}^nx^hy^h=1$ where the term $x^ky^k$ becomes ${\rm sgn}(x^k)y^k$). $M$ also has a second foliation, which is defined by the quotient of the $2n$-plans $span\{\partial/\partial x^k\}$ under the action of $\mathds{R}$ and it is $\gamma$-orthogonal to the fibers. Thus, $M$ is a locally product manifold.}\end{example}
\begin{rem}\label{obssplit} {\rm The previous example belongs to the class of {\it split generalized para-K\"ahler manifolds}. Indeed, in the representation $F=F_1+{\rm i}F_2$, $F_1=0$ on the tangent spaces of the fibers and $F_2=0$ on the orthogonal spaces, therefore, $F_1\circ F_2=F_2\circ F_1=0$.  It would be interesting to get more information about this class of manifolds. Here, we just notice the following simple fact. Recall the decomposition $TM=P\oplus Q$ (Section 2). Since $F_1^3=F_1, F_2^3=-F_2$, we get $P=P_+\oplus P_-$, where the terms are the $\pm$-eigenbundles of $F_1$, and $Q^c=E\oplus\bar{E}$, where $E$ is the $E$-eigenbundle of $F_2$. In the generalized para-K\"ahler case, the evaluation of $N_F=0$ on arguments in $P_\pm$ shows that the real subbundles $P_\pm$ are foliations. Similarly, the evaluation on $E$ shows that the complex bundles $E,\bar{E}$ are involutive, hence, if $Q$ is a foliation, the leaves are complex manifolds.} \end{rem}
\begin{example}\label{homogintegr} {\rm The invariant structure of a homogeneous space $G/H$ defined in Example \ref{omogen} is integrable iff, with the notation of the latter, for any $ad\,\mathfrak{h}$-invariant vectors $X_1,X_2,X_3\in\mathfrak{g}/\mathfrak{h}$, the following two relations hold:
$$[\tilde{F}X_1,\tilde{F}X_2]_{\mathfrak{g}}-\tilde{F}[\tilde{F}X_1,X_2]_{\mathfrak{g}}
-\tilde{F}[X_1,\tilde{F}X_2]_{\mathfrak{g}}+[X_1,X_2]_{\mathfrak{g}}=0,$$
$$\sum_{Cycl(1,2,3)}\tilde{\psi}(\tilde{X}_1,[\tilde{X}_2,\tilde{X}_3]) ={\rm i}\sum_{Cycl(1,2,3)}\tilde{\omega}(\tilde{X}_1,[\tilde{X}_2,\tilde{X}_3]).$$
In these relations the bracket $[\,,\,]_{\mathfrak{g}}$ is naturally induced by the bracket of $\mathfrak{g}$ and, using the expression of the Nijenhuis tensor and of the exterior differential, the result follows from Theorem \ref{th2integr} since $[\,,\,]_{\mathfrak{g}}$ corresponds to the Lie bracket of $ad\,\mathfrak{h}$-invariant vector fields on $M$.} \end{example}
\begin{example}\label{exlift} {\rm Let $p:TM\rightarrow M$ be a tangent bundle with local coordinates $(x^i,y^i)$ on the total space, where $(x^i)$ are coordinates on $M$ and $(y^i)$ are corresponding vector coordinates. We will denote the vertical and complete lifts of tensors from $M$ to $TM$ by upper indices $\mathbf{v},\mathbf{c}$ and recall the following definitions \cite{YI} \begin{equation}\label{deflifts}
\begin{array}{l} f^\mathfrak{v}=p^*f,\,f^\mathfrak{c}=W(f^\mathbf{v})\;(W=[y^i(\partial/\partial x^i)]_\mathfrak{V}\in T(TM)/\mathfrak{V}),\vspace*{2mm}\\ X^\mathfrak{v}f^\mathfrak{v}=0,\,X^\mathfrak{v}f^\mathfrak{c} =X^\mathfrak{c}f^\mathfrak{v}=(Xf)^\mathfrak{v},\, X^\mathfrak{c}f^\mathfrak{c}=(Xf)^\mathfrak{c},\vspace*{2mm}\\ \alpha^\mathfrak{v}(X^\mathfrak{v})=0,\,\alpha^\mathfrak{v}(X^\mathfrak{c}) =\alpha^\mathfrak{c}(X^\mathfrak{v})=(\alpha(X))^\mathfrak{v}, \, \alpha^\mathfrak{c}(X^\mathfrak{c})=(\alpha(X))^\mathfrak{c},\vspace*{2mm}\\  (P\otimes Q)^\mathfrak{v}=P^\mathfrak{v}\otimes Q^\mathfrak{v}, (P\otimes Q)^\mathfrak{c}=P^\mathfrak{c}\otimes Q^\mathfrak{v}+P^\mathfrak{v}\otimes Q^\mathfrak{c},\end{array}\end{equation}
where $\mathfrak{V}$ is the vertical bundle, tangent to the fibers, $f\in C^\infty(M),X\in \chi(M),\alpha\in\Omega^1(M)$ and $P,Q$ are arbitrary tensor fields on $M$.
On the other hand, any subbundle $S\subseteq TM$ produces a subbundle $S^\mathbf{t}\subseteq T\mathcal{T}M$, which we call the {\it tangent lift} of $S$ and has $rank\,S^\mathbf{t}=2rank\,S$. Namely, if we denote by underlining the sheaf of germs of sections of a bundle, then, $\underline{S}^\mathbf{t}=span_{\underline{M\times\mathds{R}}}\{X^\mathbf{c},X^\mathbf{v}\,/\,X\in\underline{S}\}$.

Now, assume that $M$ has a generalized almost para-Hermitian structure defined by a triple $(\gamma,\psi,F)$ and consider the triple $(\gamma^\mathfrak{c},\psi^\mathfrak{c},F^\mathfrak{c})$. Using (\ref{deflifts}) we get
\begin{equation}\label{liftstr} \begin{array}{c}
\gamma^\mathbf{c}(X^\mathbf{v},Y^\mathbf{v})=0,\,
\gamma^\mathbf{c}(X^\mathbf{c},Y^\mathbf{v})=(\gamma(X,Y))^\mathfrak{v},\vspace*{2mm}\\	 \gamma^\mathbf{c}(X^\mathbf{c},Y^\mathbf{c})=(\gamma(X,Y))^\mathfrak{c},\vspace*{2mm}\\ F^\mathbf{c}X^\mathbf{v}=(FX)^\mathbf{v},\, F^\mathbf{c}X^\mathbf{c}=(FX)^\mathbf{c},\vspace*{2mm}\\
\omega^\mathbf{c}(X^\mathbf{v},Y^\mathbf{v})=0,\,
\omega^\mathbf{c}(X^\mathbf{c},Y^\mathbf{v})=(\omega(X,Y))^\mathfrak{v},\vspace*{2mm}\\ \omega^\mathbf{c}(X^\mathbf{c},Y^\mathbf{c})=(\omega(X,Y))^\mathfrak{c},
\end{array}\end{equation}
where $\omega$ is the $2$-form associated to $(\gamma,F)$.

Using (\ref{liftstr}), it follows easily that
$(F^\mathbf{c})^2=Id$, $F^\mathbf{c}$ is $\gamma^\mathbf{c}$-skew-symmetric and the $\pm1$-eigenbundles of $F^\mathbf{c}$ are the tangent lift $S_\pm^\mathbf{t}$ of the $\pm1$-eigenbundles $S_\pm$ of $F$. Therefore, the triple $(\gamma^\mathfrak{c},\psi^\mathfrak{c},F^\mathfrak{c})$ defines a generalized almost para-Hermitian structure on the tangent manifold $TM$. Moreover, since one has $[X^\mathbf{c},Y^\mathbf{c}]=[X,Y]^\mathbf{c},
[X^\mathbf{c},Y^\mathbf{v}]= [X,Y]^\mathbf{v},[X^\mathbf{v},Y^\mathbf{v}]=0$, if the subbundles $S_\pm$ are involutive the subbundles $S_\pm^\mathbf{t}$ are involutive too. Finally, (\ref{liftstr}) also implies that the $2$-form $\omega$ of the lifted structure is just the complete lift $\omega^\mathbf{c}$. Furthermore, for any form $\lambda\in\Omega^k(M)$, if we evaluate $\lambda^\mathbf{c}$ on vertical and complete lifts, we get zero, except for
$$\lambda^\mathbf{c}(X_1^\mathbf{c},...,X_{k-1}^\mathbf{c},X_k^\mathbf{v}) =(\lambda(X_1,...,X_k))^\mathbf{v}, \lambda^\mathbf{c}(X_1^\mathbf{c},...,X_{k-1}^\mathbf{c},X_k^\mathbf{c}) =(\lambda(X_1,...,X_k))^\mathbf{c}$$ and $d\lambda^\mathbf{c}=(d\lambda)^\mathbf{c}$.
This allows us to check that, if integrability condition (\ref{necint}) holds for the structure of $M$, it also holds for the lifted structure of $\mathcal{T}M$. Therefore, if $M$ is a para-K\"ahler manifold the tangent manifold $\mathcal{T}M$ is also para-K\"ahler.}\end{example}

We end this section by a few remarks on manifolds $(M,F)$, where $F$ is a complex $(1,1)$-tensor field with involutive eigenbundles $S_\pm$ and $F^2=Id$ (including the generalized para-K\"ahler manifolds), in the footsteps of complex geometry. The functions and forms below are complex valued even though this is not apparent in notation. The decomposition $T^cM=S_+\oplus S_-$ yields a bi-grading $\Omega^{\bullet}=\oplus_{p,q\geq0}\Omega_{p,q}$, where the terms are the spaces of forms of {\it type} $(p,q)$, defined by the condition that they take the value $0$ unless evaluated on $p$ arguments in $S_+$ and $q$ arguments in $S_-$. Furthermore, one has $d=d_++d_-$, where $d_+:\Omega_{p,q}\rightarrow\Omega_{p+1,q}$, $d_-:\Omega_{p,q}\rightarrow\Omega_{p,q+1}$ and $d_\pm^2=0$, $d_+d_-+d_-d_+=0$. The bi-grading defines a double complex structure on the complex-valued de Rham complex of $M$ and spectral sequence theory may give information about de cohomology $H^{\bullet}(M,\mathds{C})$. In particular, we have corresponding $d_\pm$-cohomology spaces $^+\hspace{-2pt}H^{\bullet}_{q}(M,F),^-\hspace{-1mm}H^{\bullet}_{p}(M,F)$
defined by the complexes
\begin{equation}\label{pmcomplexes}
\begin{array}{l} \Omega_{0,q}(M)\stackrel{d_+}{\rightarrow}\Omega_{1,q}(M) \stackrel{d_+}{\rightarrow}\Omega_{2,q}(M)\stackrel{d_+}{\rightarrow}....,\vspace*{2mm}\\ \Omega_{p,0}(M)\stackrel{d_-}{\rightarrow}\Omega_{p,1}(M) \stackrel{d_-}{\rightarrow}\Omega_{p,2}(M)
\stackrel{d_-}{\rightarrow}.....\end{array}\end{equation}

For functions, we have $<d_\pm f,X_\pm>=X_\pm f,<d_\pm f,X_\mp>=0$ and we will denote $C_\pm^\infty(M)=\{f\in C^\infty(M,\mathds{C})\,/\,d_\mp f=0\}$ the algebras of {\it $F$-positive} and {\it $F$-negative} functions, respectively. Furthermore, we may define an {\it $F$-positive ($F$-negative)} vector bundle as a vector bundle $V$ endowed with an atlas of local trivializations that has $F$-positive ($F$-negative) transition functions. Then, acting on components, the operators $d_\mp$ extend to $V$-valued forms and we may define {\it $F$-positive}, respectively {\it $F$-negative} cross sections $s\in\Gamma V$ by the condition $d_{\mp}s=0$. The bundles $S_\pm$ may not be $F$-positive ($F$-negative). We can define $F$-{\it positive} ($F$-{\it negative}) vector fields $X$ on $M$ by the condition $[X'_\mp,X]\in S_\mp$, $\forall X'_\mp\in S_\mp$. Then, the field $fX$ is also $F$-positive ($F$-negative) iff $f\in C_\pm^\infty(M)$. Therefore, if $S_\pm$ has local bases of $F$-positive ($F$-negative) sections, it will be an $F$-positive ($F$-negative) bundle.

We shall indicate an interesting case where the complexes of sheaves associated to (\ref{pmcomplexes}) are exact resolutions of the sheafs of germs of $F$-negative ($F$-positive) functions, i.e., where $d_\pm$ satisfies a local Poincar\'e lemma. Recall that a complex subbundle $S\subseteq T^cM$ of rank $r$ is {\it Nirenberg integrable} if: (1) $S$ is involutive, (2) $k=dim\,S\cap\bar{S}=const.$ and $S+\bar{S}$ is involutive. Nirenberg integrability is characterized by the local existence of real, differentiable functions $y^a,t^u$, $a=1,...,k$, $u=1,...,dim\,M-2r+k$ and complex, differentiable functions $z^\alpha$, $\alpha=1,...,r-k$ such that $(y^a,t^u,z^\alpha,\bar{z}^\alpha)$ are functionally independent and $S$ has the local equations $dt^u=0, dz^\alpha=0$ \cite{Nir}.
\begin{prop}\label{Poincare} (Poincar\'e lemma) Assume that the complex tensor field $F$ ($F^2=Id$) has the following properties: (i) the two eigenbundles $S_\pm$ are Nirenberg integrable, (ii) $S_-\cap \bar{S}_+=0$. If $\lambda_{p,q}$, where $q\geq1$ satisfies the condition $d_-\lambda=0$, then, there exists a local form $\mu_{p,q-1}$ such that $\lambda=d_-\mu$. A similar result holds if we change the role of the indices $+,-$.\end{prop}
\begin{proof} As recalled above, we have the local functions $y_\pm^{a_\pm},t_\pm^{u_\pm},z_\pm^{\alpha_\pm}$ such that $ann\,S_\pm$ is generated by $(dt_\pm^{u_\pm},dz_\pm^{\alpha_\pm})$ (the signs distinguish between the two bundles and must be used concordantly). We have $r=rank\,S_\pm=n$ and, since $S_+\cap S_-=0$, the $1$-forms $(dt_-^{u_-},dz_-^{\alpha_-},dt_+^{u_+},dz_+^{\alpha_+})$ are a local basis of $T^{c*}M$, where $(dt_-^{u_-},dz_-^{\alpha_-})$ is a basis of the $(1,0)$-forms and $dt_+^{u_+},dz_+^{\alpha_+}$ is a basis of the $(0,1)$-forms. We will denote the elements of these bases by $\epsilon^i,\kappa^j$, respectively, therefore, $d\epsilon^i=0,d\kappa^j=0$.
Accordingly, we have the local expression
\begin{equation}\label{exprlambda}
\lambda=\frac{1}{p!q!}\lambda_{i_1...i_pj_1...j_q}\epsilon^{i_1} \wedge...\wedge\epsilon^{i_p} \wedge\kappa^{j_1}\wedge...\wedge\kappa^{j_q}.\end{equation}

We will continue by using the same induction as in the case of foliations \cite{VCz}, induction on the index $h$ defined as the largest among the indices of $\epsilon^1,...,\epsilon^n$ that is actually contained in (\ref{exprlambda}). We shall refer to the induction start $h=0$ at the end and, first, we show that, if the result holds for largest indices $l<h$, it also holds for $h$. If $h$ is the indicated largest index of $\lambda$, we have $\lambda=\epsilon^h\wedge\lambda'_{p-1,q}+\lambda''_{p,q}$, where the largest index of $\lambda',\lambda''$ is $<h$. Since $d\epsilon^h=0$, $d_-\lambda=0$ is equivalent to $d_-\lambda'=0,d_-\lambda''=0$, which is assumed to imply $d_-\lambda'=d_-\mu',d_-\lambda''=d_-\mu''$, whence $\lambda=-d_-(\epsilon^h\wedge\mu'+\mu'')$.

For the case $h=0$, $\lambda$ is of type $(0,q)$ and we shall need hypothesis (ii). The latter implies that there is no relation between the functions $z_-^{\alpha_-},\bar{z}_+^{\alpha^+}$. In this situation, if
$d_-\lambda_{0,q}=0$, the consequence $\lambda_{0,q}=d_-\mu_{0,q-1}$ follows from the usual Poincar\'e lemma in the space of the coordinates $(t_+^{u_+},z_+^{\alpha_+})$, with $(t_-^{u_-},z_-^{\alpha_-})$ seen as parameters. Note that, if $F=F_1+{\rm i}F_2$, then, (ii) holds if either $F_1$ is non degenerate or $F_2$ does not have eigenvalue ${\rm i}$. The result for $d_+$ will be obtained similarly.
\end{proof}

The type decomposition of $\xi\in T^{*c}M$ is $\xi=\xi_++\xi_-$, where $\xi_+=(1/2)(\xi\circ(Id+F)),\xi_-=(1/2)(\xi\circ(Id-F))$ and the symbols of $d_\pm$ are $$(\sigma(d_+)(\xi))\lambda=\xi_+\wedge\lambda,\, (\sigma(d_-)(\xi))\lambda=\xi_-\wedge\lambda, \;\xi\in T^*M,\,\forall\lambda\in\Omega^{\bullet}(M).$$

For $\xi_\pm\neq0$, we have $(\sigma(d_\pm)(\xi))\lambda=0$ iff $\exists\lambda'\in\Omega^{\bullet}(M)$	 such that $\lambda=\xi_\pm\wedge\lambda'$.
Therefore, the first, respectively the second, complex (\ref{pmcomplexes}) is elliptic iff for any real covector $\xi$, $\xi_+=0$ implies $\xi=0$, respectively, $\xi_-=0$ implies $\xi=0$. We can write this condition as $(ann\,S_\pm)\cap (T^*M)=0$, where the terms of the intersection are real subspaces of the real dual of the $4n$-dimensional space $T^{cr}M=TM\oplus({\rm i}TM)$ defined by the real structure of the complexification (${\rm i}TM$ is also seen as a real vector space).  In terms of the annihilated spaces, this condition becomes $S_\pm+{\rm i}TM=T^{cr}M$, which, because the terms have real dimension $2n$, is the same as $S_\pm\cap({\rm i}TM)=0$. In a different way, for a real $\xi$, $\xi_\pm=(1/2)[\xi\pm(\xi\circ F_1+{\rm i}\xi\circ F_2)]$ and the result is zero iff $\xi\pm\xi\circ F_1=0,\xi\circ F_2=0$. These conditions imply $\xi=0$ for $\xi\in (ker\,F_2^*)\cap(ann\,(Id\pm F_1)^*)$. In particular, if $F_2$ is an isomorphism both complexes are elliptic.

If there exists a real tangent vector $Z\neq0$ such that ${\rm i}Z\in S_\pm$, then $\overline{{\rm i}Z}=-{\rm i}Z\in S_\pm$ and ${\rm i}Z\in S_\pm\cap\bar S_\pm$. Therefore, if we ask $S_\pm\cap\bar S_\pm=0$, we get an elliptic complex. $S_\pm\cap\bar S_\pm=0$ means that $S_\pm$ is the ${\rm i}$-eigenbundle of some almost complex structure $\Phi$ ($\Phi^2=-Id$) on $M$, i.e.,
$$F(X-{\rm i}\Phi X)=\pm(X-{\rm i}\Phi X),\,\Phi(X\pm FX)={\rm i}(X\pm FX).$$
With $F=F_1+{\rm i}F_2$, these conditions reduce to
\begin{equation}\label{auxelip} \begin{array}{l}
F_2=(F_1\mp Id)\circ\Phi\,\Leftrightarrow\,F_1+F_2\Phi=\pm Id,\vspace*{2mm}\\ F_2=\mp\Phi\circ(Id\pm F_1) \,\Leftrightarrow\,F_1=\Phi\circ F_2\pm Id,
\end{array}\end{equation}
which implies $F_1\circ\Phi+\Phi\circ F_1=0$. Using that, (\ref{auxelip}) implies
$$0=F_1\circ\Phi+\Phi\circ F_1=\pm F_1^2\circ(Id+\Phi),$$
whence, composing by $(Id-\Phi)$, we get $F_1^2=0$. Furthermore, since $F_1^2-F_2^2=Id$, $F_2^2=-Id$.
Finally, if $\gamma$ is a compatible metric, i.e., $\gamma(F_1X,Y)+\gamma(X,F_1Y)=0$, then, $F_2$ is $\gamma$-skew-symmetric iff $\Phi$ is $\gamma$-skew-symmetric (with $F_2^2=-Id$, (\ref{auxelip}) implies $\Phi=-F_1F_2\pm F_2$, which has to be used to go from $F_2$ to $\Phi$).

For the structure $(\gamma,F)$ and if the complexes (\ref{pmcomplexes}) are elliptic, one may continue to a corresponding Hodge theory.
\section{Submanifolds and reduction}
Let $(M,\mathcal{F},\mathcal{J})$ be a generalized almost para-Hermitian manifold with the corresponding triple $(\gamma,\psi,F)$ and let $\iota:N\subseteq M$ be a submanifold of $M$. We will say that $N$ is a {\it regular invariant submanifold} if $\iota^*\gamma$ is non degenerate and $T^cN$ is invariant by $F$. Then, $N$ has a naturally induced generalized almost para-Hermitian structure $(\mathcal{F}_N,\mathcal{J}_N)$ defined by the triple $(\gamma_N=\iota^*\gamma,\psi_N=\iota^*\psi,F_N=F|_{T^cN})$. All the objects related to the induced structure will be denoted by the index $N$.

We recall the general pullback operation (e.g., \cite{BR})
\begin{equation}\label{defback}
\overleftarrow{f}^*U=\{(X,f^*\eta)\,/\,(f_*X,\eta)\in U\}\subseteq V\oplus V^*,
\end{equation}
where $f_*:V\rightarrow W$ is a linear mapping of vector spaces, $f^*$ is the transposed mapping and $U$ is a subspace of $W\oplus W^*$.

For any submanifold, taking $f_*=\iota_*:T_xN\subseteq T_xM$ $(x\in N)$, formula (\ref{tau}) gives
$$\mathbf{H}_N=\{(X,\flat_{\iota^*\psi+{\rm i}\iota^*\gamma}X)\,/\,X\in TN\}
=\{(X,\iota^*(\flat_{\psi+{\rm i}\gamma}X))\,/\,X\in TN\}=\overleftarrow{\iota}^*\mathbf{H}.$$
In the same way, since we know that $\mathbf{H}_\pm=\tau(S_\pm)$, we get
$\mathbf{H}_{N\pm}=\overleftarrow{\iota}^*\mathbf{H}_\pm$,
which leads to the equalities $\mathbf{F}_{N\pm}=\overleftarrow{\iota}^*\mathbf{F}_\pm,
\,\mathbf{J}_{N}=\overleftarrow{\iota}^*\mathbf{J}$. Hence,	 $\mathcal{F}$ induces a generalized almost para-complex structure and $\mathcal{J}$ induces a generalized almost complex structure of $N$ in the sense of \cite{{BB},{IV}}.
\begin{prop}\label{sbvarind} $\iota:N\subseteq M$ is a regular invariant submanifold iff
\begin{equation}\label{descsbvar} \mathbf{T}^cN=\overleftarrow{\iota}^*\mathbf{H}_+ \oplus\overleftarrow{\iota}^*\mathbf{H}_-\oplus \overleftarrow{\iota}^*\bar{\mathbf{H}}_+ \oplus\overleftarrow{\iota}^*\bar{\mathbf{H}}_-
\end{equation} defines a generalized almost para-Hermitian structure on $N$. In this case, the structure defined by (\ref{descsbvar}) is the induced structure $(\mathcal{F}_N,\mathcal{J}_N)$ defined above.\end{prop}
\begin{proof} Since, if the induced structure exists, $\mathbf{H}_{N\pm}=\overleftarrow{\iota}^*\mathbf{H}_\pm$, (\ref{descsbvar}) is the decomposition (\ref{descptparaK}) defined by the induced structure. For the converse, let us denote by a tilde the elements of the generalized almost para-Hermitian structure defined by decomposition (\ref{descsbvar}). Definition (\ref{defback}) of the pullback operation gives
$$\tilde{\mathbf{H}} =\overleftarrow{\iota}^*\mathbf{H}_+ \oplus\overleftarrow{\iota}^*\mathbf{H}_-=\tau(TN^c\cap{S}_+) \oplus\tau(TN^c\cap{S}_-),$$ where $\tau$ is defined on $N$ by $\tau(X)=(X,\flat_{\iota^*\psi+{\rm i}\iota^*\gamma})$.
This equality and the definition of $\tilde{\tau}$ yield $\tilde{\psi}=\iota^*\psi,\tilde{\gamma}=\iota^*\gamma$. In particular, $\iota^*\gamma$ is non degenerate because $\tilde{\gamma}$ is non degenerate.

Furthermore, we have
$$\tilde{\mathbf{F}}_\pm=\overleftarrow{\iota}^*\mathbf{H}_\pm \oplus\overleftarrow{\iota}^*\bar{\mathbf{H}}_\pm=\tau(T^cN\cap{S}_\pm)
\oplus\tau(T^cN\cap\bar{S}_\pm).$$
Then, since $dim\,\tilde{\mathbf{F}}_\pm=dim\,N$, we must have $T^cN=(TN^c\cap{S}_\pm)\oplus(T^cN\cap\bar{S}_\pm)$, whence $T^cN$ is invariant by $F$ and $\tilde{F}=F|_{T^cN}$.
\end{proof}
\begin{corol}\label{sbintegr} If $M,\mathcal{F},\mathcal{J})$ is a generalized para-K\"ahler manifold and $N$ is a regular invariant submanifold, $N$ with the induced structure is again a generalized para-K\"ahler manifold.\end{corol}
\begin{proof} The integrability of the structure of $M$ is equivalent with the fact that $\mathbf{H_\pm}$ are closed under the Courant bracket. The corresponding bundles of the induced structure are $\overleftarrow{\iota}^*\mathbf{H_\pm}$ and a result of \cite{C} tells that the pullback bundles are also closed under the Courant bracket.\end{proof}

In order to give another result, we recall a way to see the pullback $\overleftarrow{\iota}^*$ used in \cite{BS}. Consider the bundle $B_N=TN\oplus T^*M|_N\subseteq\mathbf{T}M$ and the projection
$$s:B_N\rightarrow B_N/(ann\,TN)\approx \mathbf{T}N.$$
Then, for any subbundle $U\subseteq \mathbf{T}^cM$, we get $\overleftarrow{\iota}^*U=s(U\cap B^c)$ at each point of $N$ (the field of subspaces $\overleftarrow{\iota}^*U$ may not be a smooth subbundle). Now, as in the generalized almost Hermitian case \cite{BS} we prove
\begin{prop}\label{FJinv} The submanifold $N$ is a regular invariant submanifold iff, $\forall x\in N$, the structures $\mathcal{F}_x,\mathcal{J}_x$ induce a generalized paracomplex, respectively complex, structure on $T_xN$ and $B=B\cap(\mathcal{F}B)\cap(\mathcal{J}B)\cap(\mathcal{H}B)+ann\,TN$.
\end{prop}
\begin{proof} For any submanifold $N$ of $M$, let us denote $C=B\cap(\mathcal{F}B)\cap(\mathcal{J}B)\cap(\mathcal{H}B)$, equivalently, $\mathcal{X}\in\mathbf{T}M$ belongs to $C$ iff $\mathcal{X},\mathcal{F}\mathcal{X},\mathcal{J}\mathcal{X},\mathcal{H}\mathcal{X}\in B$. Since $\mathbf{H}_\pm=\mathbf{F}_\pm\cap\mathbf{H}$, $\mathcal{X}\in\mathbf{T}^cM$ has the projections
$$ pr_{\mathbf{H}_\pm}\mathcal{X}=\frac{1}{4}(Id\pm\mathcal{F})(Id-{\rm i}\mathcal{H})\mathcal{X}$$
and we see that $\mathcal{X}\in C$ ($\mathcal{X}$ is real) iff $pr_{\mathbf{H}_\pm}\mathcal{X}\in B^c$. This result implies the equality
$$C^c=(B^c\cap\mathbf{H}_+)\oplus(B^c\cap\mathbf{H}_-)\oplus (B^c\cap\bar{\mathbf{H}}_+)\oplus(B^c\cap\bar{\mathbf{H}}_-),$$
hence,
\begin{equation}\label{auxJFinv}
s(C^c)=\overleftarrow{\iota}^*\mathbf{H}_+\oplus \overleftarrow{\iota}^*\mathbf{H}_-\oplus
\overleftarrow{\iota}^*\bar{\mathbf{H}}_+ \oplus\overleftarrow{\iota}^*\bar{\mathbf{H}}_-.
\end{equation}

If $N$ is a regular invariant submanifold, Proposition \ref{sbvarind} tells us that the right hand side of (\ref{auxJFinv}) is $\mathbf{T}^cN$ and $B=C+ann\,TN$. On the other hand, we already know that a regular invariant submanifold has structures induced by $\mathcal{F},\mathcal{J}$.

Conversely, if $B=C+ann\,TN$, $s(C^c)=\mathbf{T}^cN$ and (\ref{auxJFinv}) is a decomposition of the form (\ref{descsbvar}). Moreover, if structures induced by $\mathcal{F},\mathcal{J}$ exist,
$\overleftarrow{\iota}^*\mathbf{H}_\pm\oplus
\overleftarrow{\iota}^*\bar{\mathbf{H}}_\pm, \overleftarrow{\iota}^*\mathbf{H}_+\oplus
\overleftarrow{\iota}^*\bar{\mathbf{H}}_-$ are the maximal isotropic subbundles that define the induced structures. Accordingly, (\ref{auxJFinv}) defines an induced, generalized almost para-Hermitian structure on $N$.
\end{proof}
\begin{corol}\label{totalinv} Any submanifold $\iota:N\subseteq(M,\mathcal{F},\mathcal{J})$ such that $\mathcal{F}B=B$ is a regular invariant submanifold.\end{corol}
\begin{proof} A submanifold with the indicated property will be called {\it strongly $\mathcal{F}$-invariant} and we will see that $\mathcal{F}$ induces a generalized almost para-complex structure on $N$. Indeed, the necessary and sufficient condition for that are \cite{IV}
\begin{equation}\label{auxIV}
TN\cap\sharp_\phi(ann\,TN)=0,\; P(TN)\subseteq TN+im\,\sharp_\phi,
\end{equation}
where $P,\phi$ are entries of the matrix (\ref{matrici}) of $\mathcal{F}$. Since strong invariance is equivalent to $P(TN)\subseteq TN,im\,\sharp_\phi\subseteq TN$ the second condition (\ref{auxIV}) obviously holds. On the other hand, we have $$<\nu,\sharp_\phi\lambda>=
-<\lambda,\sharp_\phi\nu>=0,\,\forall\nu\in ann\,TN,\lambda\in T^*M|_N,$$
whence $\sharp_\phi(ann\,TN)=0$ and the first condition (\ref{auxIV}) holds too.

Furthermore, we prove that every submanifold satisfies the property $B=B\cap(\mathcal{H}B)+ann\,TN$. Indeed, by looking at the projections $pr_{\mathbf{H}}\mathcal{X}=
(1/2)(\mathcal{X}-{\rm i}\mathcal{H}\mathcal{X}), pr_{\bar{\mathbf{H}}}\mathcal{X}=
(1/2)(\mathcal{X}+{\rm i}\mathcal{H}\mathcal{X})$, $\mathcal{X}\in B\cap(\mathcal{H}B)$, we see that
$$s(B\cap(\mathcal{H}B))^c=s(\mathbf{H}\cap B^c)\oplus s(\bar{\mathbf{H}}\cap B^c)=
\overleftarrow{\iota}^*\mathbf{H}\oplus\overleftarrow{\iota}^* \bar{\mathbf{H}}=\mathbf{T}^cN=
s(B^c).$$

Then, from $\mathcal{F}B=B$, we deduce
\begin{equation}\label{BcapHB} \begin{array}{l}
B=B\cap(\mathcal{H}B)+ann\,TN=B\cap(\mathcal{J}(\mathcal{F}B))+ann\,TN\vspace*{2mm}\\ \hspace*{3mm}=B\cap(\mathcal{J}B)+ann\,TN,\end{array}
\end{equation} 
which characterizes the fact that the tangent spaces $T_xN$, $x\in N$ have a generalized almost complex structure induced by $\mathcal{J}$ (this was proven in \cite{BS}). On the other hand, we also have $B\cap(\mathcal{F}B)\cap(\mathcal{J}B)\cap(\mathcal{H}B)=B\cap(\mathcal{J}B)$. Together with (\ref{BcapHB}), this shows that the hypotheses of Proposition \ref{FJinv} hold.\end{proof}

The second subject of this section is reduction. Let $(M,\mathcal{F},\mathcal{J})$ be a generalized almost para-Hermitian manifold with the corresponding triple $(\gamma,\psi,F)$. Assume that a Lie group $G$ acts on $M$ such that the structure is preserved. The last assertion has the following meaning: if we see $g\in G$ as a diffeomorphism of $M$ and define its action on $\mathbf{T}M$ by
$$g_\bigstar(X,\alpha)=(g_*X,g^{-1*}\alpha),$$ then, $\mathcal{F}$ and $\mathcal{J}$ commute with $g_\bigstar$. This is easily seen to be equivalent with the fact that $g_\bigstar$ preserves the subbundles $\mathbf{H}_\pm$, which, further, is equivalent with the preservation of the triple $(\gamma,\psi,F)$ (preservation of $F$ means commutation with the differentials $g_*$).

Reduction is a process originating in symplectic geometry and, in the present case, we shall define it with reference to the complex $2$-form $\omega(X,Y)=\gamma(X,FY)$ as in the similar, para-Hermitian case \cite{VJMP}. A submanifold $\iota:N\subseteq M$ that is invariant by a subgroup $G'\subseteq G$ is said to be a {\it reducing submanifold} if (i) the action of $G'$ on $N$ is free and proper, (ii) $\forall x\in N$, the orbit $G(x)$ cleanly intersects $N$ and $T^cG(x)=(T^cN)^{\perp_\omega}$, (iii) the submanifold $G(x)\cap N$ is equal to the orbit $G'(x)$.
\begin{prop}\label{redcuG} Let $N$ be a reducing submanifold of the generalized almost para-Hermitian manifold $(M,\gamma,\psi,F)$. Assume that the metric $\gamma$ is non degenerate on $N$ and on the orbits $G'(x)$ and that $ann\,\psi_x\supseteq T_xG'(x)$ ($x\in N$). Then, the quotient manifold $Q=N/G'$  has a generalized almost para-Hermitian structure defined by the structure of $M$.\end{prop}
\begin{proof}
Since the intersection $G(x)\cap N$ is clean,  if we denote $K=ann(\iota^*\omega_x)$, we get
\begin{equation}\label{aux1red}
T^c(G'(x))=T^cG(x)\cap T^cN=(T^cN)^{\perp_\omega}\cap T^cN=K\;\;(\forall x\in N)
\end{equation}
and $rank(\iota^*\omega)=const.$ We will denote by $\mathcal{K}$ the manifold given by the sum of the orbits of $G'$.
Condition (i) ensures the existence of the quotient manifold $Q$ and of the natural submersion $p:N\rightarrow Q$ that sends $x$ to the orbit $G'(x)$. By (\ref{aux1red}), we have
$i(X)(\iota^*\omega_x)=0$, $\forall X\in T^c(G'(x))$. Furthermore, any local vector field $X\in T\mathcal{K}$ is a finite linear combination of infinitesimal transformations given by elements $\xi$ of the Lie algebra $\mathfrak{g}'$ of $G'$ and, since $G'$ preserves $\omega$ and $N$, we must have $L_X(\iota^*\omega)=0$. The three facts: $p$ submersion, $i(X)(\iota^*\omega)=0,L_X(\iota^*\omega)=0$ ensure the existence of a non degenerate complex $2$-form $\omega'$ on $Q$ such that $p^*\omega'=\iota^*\omega$ and $\omega'$ is called the {\it reduction} of $\omega$.

Then, the hypotheses on $\gamma$ imply the existence of the subbundle $V=(T\mathcal{K})^{\perp_{\iota^*\gamma}}$ with the non degenerate metric $\iota^*\gamma|_V$, which is invariant by the infinitesimal transformations of $\mathfrak{g}'$. Accordingly, we must have $\iota^*\gamma|_V=p^*\gamma'$ where $\gamma'$ is a non degenerate metric of $Q$. Furthermore, by (ii), for $x\in N$, we have $T^cN=(T^cG(x))^{\perp_\omega}$ and, since $F$ is compatible with $\omega$ and commutes with the infinitesimal transformation of $G$, we see that $F|_N$ preserves $T^cN$. Moreover, along $N$, $F$ preserves the bundle $T\mathcal{K}$, therefore, it induces the endomorphism $F'=-\sharp_{\iota^*\gamma|_N}\circ \flat_{\iota^*\omega|_N}$ of $V^c$. We may see $F'$ as an endomorphism of $T^cQ$ that satisfies (\ref{Fgamma}). Finally, the hypotheses on $\psi$ ensure that $i(X)\iota^*\psi=0, L_X\iota^*\psi=0$, $\forall X\in TN$ (same explanation as for $\omega$), whence, the existence of a $2$-form $\psi'$ of $Q$ such that $\iota^*\psi|_V=p^*\psi'$. The triple $(\gamma',\psi',F')$ defines the required structure of $Q$, also called the {\it reduction} of the original structure of $M$ {\it via the reducing submanifold} $N$.\end{proof}
\begin{prop}\label{integrred} If $M$ is a generalized para-K\"ahler manifold and $N$ is a reducing submanifold that satisfies the hypotheses of Proposition \ref{redcuG}, the corresponding reduced manifold $Q$ is also a generalized para-K\"ahler manifold.\end{prop}
\begin{proof} By Theorem \ref{th2integr}, the integrability conditions of the structure of $M$ are given by $N_F=0$ together with the equality (\ref{necint}). The Lie bracket on $Q$ is given by
$[p_*X_1,p_*X_2]=p_*[X_1,X_2]$, where $X_1,X_2$ are vector field on $N$ that admit a projection on $Q$. Thus,for the same $X_1,X_2$, the definition of the reduction $F'$ gives
$$N_{F'}(p_*X_1,p_*X_2)=p_*N_F(X_1,X_2)=0.$$
Similarly, equality (\ref{necint}) on $Q$ may be identified with its pullback by $p^*$ to $N$, hence it is implied by (\ref{necint}) on $M$.
\end{proof}
\begin{prop}\label{redcuH} Under the hypotheses of Proposition \ref{redcuG}, the subbundles $\mathbf{H}'_\pm$ of the reduced manifold $Q$ are given by $\mathbf{H}'_\pm=\overrightarrow{p}_*\overleftarrow{\iota}^*\mathbf{H}_\pm$, where $\mathbf{H}_\pm$ define the generalized almost para-Hermitian structure of $M$.\end{prop}
\begin{proof} In the proposition, $\overrightarrow{p}_*$ denotes the push-forward operation defined by (e.g., \cite{BR})
\begin{equation}\label{defpush}
\overrightarrow{p}_*U=\{(p_*X,\lambda)\in\mathbf{T}^cQ\,/\,(X,p^*\lambda)\in U\}\hspace{2mm}(U\subseteq \mathbf{T}^cN).
\end{equation}
Since $F(T^cN)\subseteq T^cN$ (see the proof of Proposition \ref{redcuG}), the eigenbundles of $F|_{T^cN}$ are $S_\pm\cap T^cN$, where $S_\pm$ are the eigenbundles of $F$. Consequently, the eigenbundles of $F'$ on $Q$ are $p_*(S_\pm\cap T^cN)$. From $\mathbf{H}_\pm=\tau S_\pm$ and the definition (\ref{defback}), we get $$\overleftarrow{\iota}^*\mathbf{H}_\pm=\{(Y,\flat_{\iota^*\psi+{\rm i}\iota^*\gamma})\,/\,Y\in S_\pm\cap T^cN\}.$$ Then, (\ref{defpush}) and the definition of the mapping $\tau$ on $Q$ lead to the required result.
\end{proof}

In particular, we can define a version of the Marsden-Weinstein reduction. The real vector field $X$ of $M$ will be called a {\it Hamiltonian vector field} if $L_X\omega=0$ and there exists a complex valued function $f$ such that $df=i(X)\omega$. (If $d\omega=0$, the first condition follows from the second. Otherwise, it may be very restrictive \cite{Valmsympl}.) The structure-preserving action of $G$ on $(M,\mathcal{F},\mathcal{J})$ is a {\it Hamiltonian action} if the infinitesimal transformations $\xi_M$ are Hamiltonian for all $\xi\in\mathfrak{g}$. Furthermore, an {\it equivariant momentum map} is an equivariant mapping $\mu:M\rightarrow\mathfrak{g}^{*c}$ (i.e., $\forall g\in G$, $\mu(g(x))=(coad\,g)(\mu(x))$) such that, $\forall\xi\in\mathfrak{g}$, $$L_{\xi_M}\omega=0,\;i(\xi_M)\omega=d\mu^\xi,\hspace{3mm}\mu^\xi(x) =<\mu(x),\xi>,\;x\in M.$$

If $\mu$ is an equivariant momentum map, a level set $N=\mu^{-1}(\theta)$, where $\theta\in\mathfrak{g}^{*c}$ is a non-critical value of $\mu$ may be a reducing submanifold of $M$. Indeed, take the subgroup $G'=G_\theta$, where $G_\theta$ is the isotropy subgroup of $\theta$ for the coadjoint action of $G$. By the equivariance of $\mu$, $N$ is $G'$-invariant and, for all $x\in N$, the orbits satisfy condition (iii) of a reducing submanifold, $G_\theta(x)=G(x)\cap N$. Furthermore, at $x\in N$, take the tangent vector $\xi_M(x)$ of $G(x)$ ($\xi\in\mathfrak{g}$) and the complex tangent vector $X$ of $N$. Then,
$$\omega(\xi_M(x),X)=<i(\xi_M(x))\omega,X>=<d_x\mu^\xi,X>=
<d_x\mu(X),\xi>=0,$$
whence we see that $X\perp_\omega T_xG(x)$ is equivalent to $d_x\mu(X)=0$. Accordingly, $T^cG(x)=(T^cN)^{\perp_\omega}$ and condition (ii) of a reducing submanifold holds. If we ask the action of $G_\theta$ on $N$ to be proper and free (condition (i)), then, $N$ is a reducing submanifold. Furthermore, if the other hypotheses of Proposition \ref{redcuG} are added, we get a reduction that may be seen as a Marsden-Weinstein reduction.
\begin{example}\label{exmoment} {\rm Let us come back to Example \ref{exlift} and assume that the connected Lie group $G$ acts on $M$ and preserves $(\gamma,\psi,F)$. The differential of this action defines an action of $G$ on the manifold $TM$. Since for any tensor field $P$ the Lie derivative has the property $L_{X^\mathbf{c}}P^\mathbf{c}=(L_XP)^\mathbf{c}$ \cite{YI}, we have $L_{\xi_{TM}^\mathbf{c}}\gamma^\mathbf{c}=0,L_{\xi_{TM}^\mathbf{c}}\psi=0, L_{\xi_{TM}^\mathbf{c}}F=0$ and the action on $TM$ preserves the complete lift of the structure. We define a momentum map $\mu:TM\rightarrow\mathfrak{g}^{*c}$ in the following way
\begin{equation}\label{muTM} <\mu(x,y),\xi>=\mu^\xi(x,y)=\omega^\mathbf{c}(\xi_{M}^\mathbf{c},E) \;((x,y)\in TM,\xi\in\mathfrak{g}^c),\end{equation}
where $E=y^i(\partial/\partial y^i)$ is the so-called Euler vector field. Putting $\alpha=i(\xi_M)\omega=\alpha_idx^i$ and taking into account the equality $i(\xi_{M}^\mathbf{c})\omega^\mathbf{c}=(i(\xi_M)\omega)^\mathbf{c}$ (which may be checked on vertical and complete lifts), we get
$$\mu^\xi=\alpha^\mathbf{c}(E)=y^k\alpha_k,\; d\mu^\xi=y^kd\alpha_k+\alpha_kdy^k =\alpha^\mathbf{c}=i(\xi_M^\mathbf{c})\omega^\mathbf{c}$$
as required for a momentum map. The $coad\,\mathfrak{g}$-equivariance of $\mu$ is a straightforward consequence of the fact that the action of $G$ on $TM$ preserves $\omega^\mathbf{c}$ and, also, preserves $E$ (for any vector field $X^\mathbf{c}$ one has $[X^\mathbf{c},E]=0$, hence $L_{\xi^\mathbf{c}}E=0$, $\forall\xi\in\mathfrak{g}$).
}\end{example}

\hspace*{7cm}{\small \begin{tabular}{l} Department of
Mathematics\\ University of Haifa, Israel\\ E-mail:
vaisman@math.haifa.ac.il \end{tabular}}
\end{document}